\documentclass[preprint,11pt]{imsart}
\RequirePackage[T1]{fontenc}
\usepackage{color}
\usepackage[latin1]{inputenc}
\usepackage[T1]{fontenc}

\usepackage{amsmath,amssymb,graphicx}
\usepackage{epstopdf}
\usepackage{color}
\RequirePackage{amsfonts,amsthm,amsmath}
\RequirePackage[colorlinks,citecolor=blue,urlcolor=blue]{hyperref}
\usepackage{a4wide}
\usepackage{here}
\usepackage{float}
\usepackage{amsmath}
\usepackage{bbm}
\usepackage{mathrsfs} 
\allowdisplaybreaks

\newcommand{\nbR}{\mathbb{R}}
\newcommand{\nbN}{\mathbb{N}}

\newcommand{\nbC}{\mathbb{C}}
\newcommand{\nbS}{\mathbb{S}}
\newcommand{\nbu}{\mathbbm{1}}

\newcommand{\nbP}{\mathbb{P}}

\newcommand{\nbE}{\mathbb{E}}

\newcommand{\Z}{{\cal Z}}
\newcommand{\M}{{\cal M}}

\newcommand{\ka}{\kappa}

\newtheorem{theorem}{{\bf Theorem}}
\newtheorem{lemm}[theorem]{{\bf Lemma}}
\newtheorem{cor}[theorem]{{\bf Corollary}}

\newtheorem{example}[theorem]{Example}
\newtheorem{prop}[theorem]{\bf Proposition}

\newcommand{\scrL}{\mathscr{L}}

\begin{document}

\begin{frontmatter}

\title{Analysis of the infinite server queues with semi-Markovian multivariate discounted inputs
}
\runtitle{Infinite server queues with semi-Markovian multivariate input}

\begin{aug}
\author{\fnms{\Large{Landy}}
\snm{\Large{Rabehasaina}}
\ead[label=e2]{lrabehas@univ-fcomte.fr}} \and \
\author{\fnms{\Large{Jae-Kyung}}
\snm{\Large{Woo}}
\ead[label=e1]{j.k.woo@unsw.edu.au}}
\runauthor{L.Rabehasaina and J.-K.Woo}
%\affiliation{University Bourgogne Franche-Comt\'e}

\address{\hspace*{0cm}\\
Laboratory of Mathematics, University Bourgogne Franche Comt\'e,\\
16 route de Gray, 25030 Besan\c con cedex, France.\\[0.2cm]
%\hspace*{0cm} \\
\printead{e2}}

\address{\hspace*{0cm}\\
School of Risk and Actuarial Studies, Australian School of Business,\\ University of New South Wales, Australia.\\[0.2cm]
%\hspace*{0cm} \\
\printead{e1}}
\end{aug}

\vspace{0.5cm}

\begin{abstract}
We consider a general $k$ dimensional discounted infinite server queues process (alternatively, an Incurred But Not Reported (IBNR) claim process) where the multivariate inputs (claims) are given by a $k$ dimensional finite state Markov Chain and the arrivals follow a renewal process. After deriving a multidimensional integral equation for the moment generating function jointly to the state of the input at time $t$ given the initial state of the input at time 0, asymptotic results for the first and second (matrix) moments of the process are provided. In particular, when the interarrival or service times are exponentially distributed, transient expressions for the first two moments are obtained. Also, the moment generating function for the process with deterministic interarrival times is considered to provide more explicit expressions. Finally, we demonstrate the potential of the present model by showing how it allows us to study a semi-Markov modulated infinite server queues where the customers (claims) arrival and service (reporting delay) times depend on the state of the process immediately before and on the switching times.

\end{abstract}
%\begin{keyword}[class=AMS]
%\kwd[Primary ]{60G50}
%\kwd{60K30}
%\kwd{62P05}
%\kwd{60K25}
%\kwd[; secondary ]{91B84}
%\kwd{62P05}
%\end{keyword}

\end{frontmatter}

\normalsize

%\section{Introduction}\label{sec:intro}
%
%The present paper presents . In actuarial science, it is more realistic to model (...) 

\section{Introduction}\label{sec:intro}
% . In actuarial science, it is more realistic to model (...) First, we consider 
 
In the context of an infinite server queue with correlated batch arrivals, the total number of customers still in the system is related to an aggregation of correlated risks (multivariate risks) where the arrival times of those risks are adjusted by adding a random delay. Without this time delay, probability modeling of aggregate risk processes has been studied in various areas such as applied probability, reliability theory, and actuarial science. In particular, the research on an aggregation of correlated risks is striving to develop techniques to estimate the combined effect of different types of risks on the infrastructure or system. 

In an infinite server queue with correlated batch arrivals, the random batch size is multivariate and the service time distribution is dependent on the type of input (queue). A similar idea but with multiple Markovian batch arrival streams can be found in \cite{MT02} where a time-dependent matrix joint generating function of the number of customers in the system was derived. In a renewal process with correlated batch arrivals, \cite{W16} provides the transient expressions for the joint moments of the number of customers in an infinite server queue which are recursively obtainable. Then, \cite{RW17} develop asymptotic approximation methods to study these joint moments and also provided some queueing theoretic applications, including the workload of the queue and infinite server queues in tandem.

It is worthwhile to mention that research on the actuarial application to the Incurred But Not Reported (IBNR) claims exists in much the same way as an analysis of the infinite server queues with batch arrivals, see e.g. \cite{W90, WD01, W16}. As discussed in \cite{K74}, the service times in an infinite server queue can be interpreted as the time lag between the occurrence of a claim and the report of that claim in the IBNR model. In this respect, the proposed model in the present paper can be utilized to analyze different quantities of interest in (at least) two areas including queueing theory and actuarial science.

The present paper considers the model which is an extension of the one given in \cite{RW17}. In each batch arrival, the model consists of multivariate queues (claims) which are modeled by some finite Markov Chain, and a renewal process is assumed for batch arrivals. The Markovian assumption for a vector of queues (claims) enable us to study the infinite server queues (IBNR process) in more realistic situations such that arrival times and service times are dependent on the states of the external semi-Markovian process immediately before and on the switching times, as will be illustrated later in Section \ref{sec:application}. 
It is natural to model that the arrival process and service time are modulated by some external process, in particular, when this process impacts on the intensity of claim arrival processes and in turn, the type of service times. For example, the number of multiple types of claims in catastrophe insurance varies depending on the environmental condition and also it could lead to different types of reporting/settlement time delays.
In addition, this model provides the capability to incorporate the information of successive states of the external background process before and on the switching times, which accommodates a link between consecutive states around the jump times. Moreover, when the batch sizes are regarded as claim amounts a discount factor introduced in the model is certainly important to study a present value of total IBNR claim amounts. Also, as shown in \cite{RW17} it allows us to analyze the workload and covariance of the workload and queue size, and reinterpret Little's law in the presence of a positive discount factor.

To the best of our knowledge, there is no study of the current setting of the model (especially in the presence of a discount factor) in the literature of queueing theory or actuarial science. Instead, similar settings of the model such as Markov modulated infinite server systems are found. For example, in \cite{BKMT14} the particle arrives according to a Poisson process and the Poisson arrival rate and the distribution of service times are dependent on the state of an external Markov process (background process). When the interarrival times in our model is exponential, the one-dimensional case in Section \ref{sec:application} is similar to the one studied in \cite{BKMT14}. In a system with multiple infinite server queues, \cite{MDT16} consider that both the arrival rates and the parameter of the exponentially distributed service times are modulated by a common background process. In \cite{BDTM17}, a similar model but a single queue with a Poisson arrival is revisited to study the asymptotic behavior of the number of customers in the system in the large-deviations regime. In some papers, arrival and service rates in an infinite server queue are governed by an external semi-Markov process. See \cite{D08} and \cite{FA09} for instance. \cite{OP86} study $M/M/\infty$ queue model modulated by an external continuous-time Markov Chain. 
%In addtion, various IBNR models under different stochastic and/or distributional assumptions have been studied by \cite{GLW13, LWX14, LWX17,LH13,W90} for example.

%{\bf Organization of the paper.} 
The remainder of the paper is structured as follows: In Section \ref{sec:model}, we provide a description of the mathematical model.  After deriving some general results on the (joint) moment generating function (mgf)/Laplace Transform (LT) and the first two moments in Section \ref{general_results}, we show in Section \ref{sec:particular_cases} that the limiting second-order joint moments are explicitly available when  the service times are (potentially degenerated) exponentially distributed and the transient moments are obtainable when the interarrival times are exponentially distributed. Some numerical illustrations for the limiting behaviour of the first and second joint moments are provided at the end of Section \ref{subsec:expo_service_times}. Section \ref{sec:deterministic_interarriva} is concerned with the particular case of deterministic interarrival times, where we show that the mgf has a simple expression as a matrix product (see Theorem \ref{theo_deterministic_psi}). Application to a model related to the infinite server queues is provided in Section \ref{sec:application}. It is assumed that the infinite server queues are modulated by an external semi-Markovian process, such that arrival times and service times depend on the state of the modulating process immediately before and after it switches states.

\section{Model description}\label{sec:model}
Let $\{ N_t,\ t\ge 0\}$ be a renewal process associated with a non decreasing sequence $(T_i)_{i\in\nbN}$ with $T_0=0$, such that $(T_i-T_{i-1})_{i\ge 1}$ is independent and identically distributed (iid). Also let $\tau=T_1$ with a cumulative distribution function (cdf) $F(x)=1-\bar{F}(x)$ and the LT ${\cal L}^\tau (u)=\nbE(e^{-u\tau })$ for $u\geq 0$. We introduce a stationary ergodic finite Markov Chain $(X_i)_{i\in\nbN}$ with a state space ${\cal S}=\{ 0,...,K\}^k$ for some $K\in \nbN$ and $k\in\nbN^*=\nbN\!\setminus\!\{0\}$, so that $X_i$ is for all $i$ of the form $X_i=(X_{i1},...,X_{ik})$ with $X_{ij}\in \{ 0,...,K\}$ for $j=1,...,k$. Then for $\alpha\ge 0$, the discounted process $\{Z(t)=Z(t;\alpha)\in \nbR^k ,\ t\ge 0 \}$ is a vector of $k$ processes $Z(t)=(Z_1(t),...,Z_k(t))$ with each process defined as
\begin{equation}\label{def_Z_t}
Z_j(t):=\sum_{i=1}^{N_t}X_{ij}e^{-\alpha (L_{ij}+T_i)}\nbu_{[t<L_{ij}+T_i]}=\sum_{i=1}^{\infty}X_{ij}e^{-\alpha (L_{ij}+T_i)}\nbu_{[T_i\le t<L_{ij}+T_i]},
\end{equation}
where $(L_{ij})_{i\in \nbN, j=1,...,k}$ is a sequence of independent random variables (rvs) such that $(L_{i1},...,L_{ik})_{i\in \nbN}$ is iid (although $L_{i1}$,...,$L_{ik}$ may have different distributions). We set $(L_1,...,L_k)$ to be a generic random vector distributed as the $(L_{i1},...,L_{ik})$'s, with each $L_j$ having the LT denoted by ${\cal L}_j(u)=\nbE(e^{-u L_j})$ for $u\ge 0$. As in \cite{RW17}, we let $\tilde{Z}(t)=\tilde{Z}(t;\alpha):=e^{\alpha t}Z(t;\alpha)$. The processes described in \eqref{def_Z_t} are viewed as different quantities of interest in the following two areas. 
 In queueing theory, and especially when $\alpha=0$, $X_{ij}$ represents the number of customers arriving in queue $j\in\{1,...,k\}$ at time $T_i$, each of those customers with same service time $L_{ij}$. In actuarial science, when the severities of claims of different types occurring due to a common accident or catastrophe event and there are some time delays for insurers to hear (or settle) these claims, $X_{ij}$ represents amounts of $j$-type of claim arriving at the time $T_i$ and $X_i$ is a vector of multivariate claims caused by this $i$th event. Since the present value of each claim amount is calculated by discount claim amounts $X_{ij}$ at the actual realization time $T_i+ L_{ij}$, $Z_j(t)$ in (\ref{def_Z_t}) is regarded as discounted IBNR amounts of $j$-type of claim by time $t$ and $Z(t)$ is a vector of multivariate discounted IBNR claim processes involving $k$ types of claims. Hence, $X_{ij}$ and $L_{ij}$ will in what follows be invariably referred to the claim/batch sizes and delay/service times respectively. In particular, when $\alpha=0$, this model is a generalization of the Model II in \cite{MDT16} which considers the case where $\tau$ is exponentially distributed, i.e. when the set of queues is modulated by a common continuous-time Markov Chain, as will be discussed in Section \ref{sec:particular_cases}.

A few words describing the technical difficulties that stand out against those in \cite{RW17}. Although the structure of the latter paper is similar to the present one, the correlation structure of the $X_i$'s (namely, forming a Markov Chain) considerably increases the technical challenges leading to the results proved here. For example, the iid assumption on $X_i$'s enables us to determine all joint moments of $Z_1(t),...,Z_k(t)$ recursively in \cite{RW17}, whereas some matrix issues result in the expressions only available for the first two first joint moments in this paper, see Theorems and \ref{theo_second_limit_moment_expo} and \ref{theo_second_moment_Poisson}. Besides, contrarily to \cite{RW17}, we prove in Theorem \ref{theo_deterministic_psi} that in a particular Markovian setting with constant interarrival times the distribution of the joint vector $(\tilde{Z}(t),X_{N_t})$ is converging and an explicit expression for the limiting distribution is available. Finally, the application of the model described in Section \ref{sec:application} is novel in a sense that, to the best of our knowledge, infinite server queues (IBNR process) where service time (reporting delay) for incoming customers (claims) depend on both states of the modulating exterior process at the switching time and prior to this time were not analytically studied in the literature.

{\bf Notation.} Let $P=(p(x,x'))_{(x,x')\in {\cal S}^2}$ and $\pi=(\pi(x))_{x\in {\cal S}}$ (written as a row vector) be respectively the transition matrix and stationary distribution of the Markov Chain. We next define for all $r\ge 0$ and $s=(s_1,...,s_k)\in \nbR^k $,
%\begin{equation}\label{def_pi_Q_tilda}
\begin{align}
\tilde{\pi}(s,r)&:= \displaystyle\mbox{diag}\bigg[ \nbE \bigg( \exp\bigg\{ \sum_{j=1}^k s_jx_j e^{-\alpha (L_j-r)}\nbu_{[L_j>r]}\bigg\}\bigg),\ x=(x_1,...,x_k)\in{\cal S}\bigg],\label{def_pi_Q_tilda}\\
\tilde{Q}(s,r)&:= \displaystyle\tilde{\pi}(s,r) P',\label{def_pi_Q_tilda1}
\end{align}
%\end{equation}
where $P'$ denotes the transpose of matrix $P$. We also introduce some notations in the following. $I$ is the identity matrix, ${\bf 0}$ is a column vector with zeroes, and ${\bf 1}$ is a column vector with $1$'s, of appropriate dimensions. When a random variable (rv) $X$ is exponentially distributed with mean $1/\beta$, it is denoted as $X\sim \mathcal{E}(\beta)$. Also, we let the ${\cal S}\times {\cal S}$ diagonal matrices
\begin{eqnarray}
\Delta_j&:=&\mbox{diag} \left[ x_j,\ x=(x_1,...,x_k)\in{\cal S}\right],\quad j=1,...,k,\label{Di}\\
\Delta_\pi&:=&\mbox{\normalfont diag}(\pi(x),\ x\in {\cal S}),\nonumber
\end{eqnarray}
and $\delta$ is used to denote the Kronecker symbol, e.g. $\delta_{x,y}$ equals to $1$ iff $x=y$ and $0$ else. The mgf of the process $\tilde{Z}(t)=\tilde{Z}(t;\alpha)$ jointly to the state of $X_{N_t}$ given  the initial state of $X_0$ is denoted by
\begin{equation}\label{def_mgf}
\tilde{\psi}(s,t)=\tilde{\psi}(s,t;\alpha):=\left[ \nbE\left( \left. e^{<s,\tilde{Z}(t)>}\nbu_{[X_{N_t}=y]}\right| X_0=x\right)\right]_{(x,y)\in {\cal S}^2},\quad t\ge 0,
\end{equation}
where $<\cdot, \cdot>$ denotes the scalar product on $\nbR^k$. Note that $s=(s_1,...,s_k)$ is assumed to be such that $s_j\in \nbR$ for all $j=1,...,k$ so that \eqref{def_mgf} is well defined, i.e. the expectation is finite. Definition \eqref{def_mgf} may include the case where the $s_j$'s are complex and purely imaginary, in which case $\tilde{\psi}(s,t)$ is the characteristic function of $\tilde{Z}(t)$ jointly to $X_{N_t}$; this will particularly be the case in the proof of Theorem \ref{theo_deterministic_psi}. Note also that $X_0$ in \eqref{def_mgf} has no direct physical interpretation here, as the batch sizes/claims sizes are given by $X_i$, $i\ge 1$, and is rather introduced for technical purpose. We define the first and second (matrix) moments of $ \tilde{Z}(t)$ jointly to the state of the Markov Chain $X_{N_t}$ given the initial state $X_0$ as
\begin{equation}\label{def_moments}
\begin{array}{rcl}
M_j(t)&:=&\left[\nbE \left(\left.\tilde{Z}_j(t)\nbu_{[X_{N_t}=y]}\right| X_0=x\right)\right]_{(x,y)\in {\cal S}^2},\quad j=1,... ,k ,\\
M_{jj'}(t)&:=&\left[\nbE \left(\left.\tilde{Z}_j(t)\tilde{Z}_{j'}(t)\nbu_{[X_{N_t}=y]}\right| X_0=x\right)\right]_{(x,y)\in {\cal S}^2},\quad j,j'=1,... ,k, 
\end{array}
\end{equation}
respectively. We remark that the mgf defined in (\ref{def_mgf}) is different from the one studied in \cite{RW17} which does not consider Markovian assumption for a vector $X_i$ and joint structure with the state $X_{N_t}$ conditioning on the initial state $X_0$.

This introductory section is completed by giving some results of independent interest that will be used in the rest of the paper. The following lemma is important for some computations on Markov Chains, which may be found in \cite[Lemma 1]{FG04}:
\begin{lemm}\label{lemma_compute_Mk}
Let $(S_n)_{n\in\nbN}$ be a stationary Markov Chain with a state space $E$, the transition matrix $P$ and (stationary) distribution $\pi=(\pi(x))_{x\in E}$. For all functions $f_1$,...,$f_{l+1}$ we have
\begin{equation}\label{compute_Mk}
\nbE(f_1(S_1)\cdots f_l(S_l))={\bf 1} ' \prod_{i=0}^{l-2} Q_{f_{l-i}} \pi_{f_1},
\end{equation}
where $Q_{f_i}:=\mbox{diag}(f_i(z),\ z\in E)P'$ for $i=1,...,l$, and $\pi_{f_1}:=\mbox{diag}(f_1(z),\ z\in E)\pi'$.
\end{lemm}
The following is a direct consequence of \eqref{compute_Mk}. Let $e_x$ (resp. $e_y$) be the column vector of which $z$th entry is $\delta_{x,z}$ (resp. $\delta_{y,z}$). One has then for all $x$, $y$ in $E$ that
$$
\nbE(f_1(S_1)\cdots f_l(S_l)\nbu_{[S_{l}=y]}|\ S_1=x)=e_y' \prod_{i=0}^{l-2} Q_{f_{l-i}} \mbox{diag}(f_1(z),\ z\in E) e_x , 
$$
which, because it is a scalar, is equal to its transpose, i.e.
\begin{equation*}\label{compute_Mk2}
\nbE(f_1(S_1) \cdots f_l(S_l)\nbu_{[S_{l}=y]}|\ S_1=x)=e_x' \mbox{diag}(f_1(z),\ z\in E) \prod_{i=2}^{l} Q_{f_i} ' e_y ,
\end{equation*}
which immediately implies the following corollary. 
\begin{cor}\label{coro_compute_Mk}
Under the same notation as in Lemma \ref{lemma_compute_Mk}, one has the matrix equality
\begin{equation*}\label{compute_Mk3}
\left[ \nbE(f_1(S_1)\cdots f_l(S_l)\nbu_{[S_{l}=y]}|\ S_1=x)\right]_{(x,y)\in E^2}=\mbox{\normalfont diag}(f_1(z),\ z\in E) \prod_{i=2}^{l} Q_{f_i} ' .
\end{equation*}
\end{cor}

\section{General results}\label{general_results}

\subsection{The Laplace transform}\label{sec:general_results}
The aim of this subsection is to establish some properties verified by the mgf $\tilde{\psi}(s,t)$ in (\ref{def_mgf}).
\begin{prop}\label{prop_psi_tilda}
The mgf of $\tilde{Z}(t)$ defined by \eqref{def_mgf} satisfies
\begin{equation}\label{eq_psi_tilda}
%\tilde{\psi}(s,t)=\bar{F}(t) P' \Delta_\pi +\nbE\left( \nbu_{[N_t>0]}\tilde{\pi}(s,t-T_1) \prod_{i=2}^{N_t} \tilde{Q}(s,t-T_i)'\right),
\tilde{\psi}(s,t)=\nbE\bigg(  \prod_{i=1}^{N_t} \tilde{Q}(s,t-T_i)'\bigg)=\bar{F}(t) I + \nbE\bigg( \nbu_{[N_t>0]} \prod_{i=1}^{N_t} \tilde{Q}(s,t-T_i)'\bigg)
\end{equation}
for all $s\in\nbR^k$, $t\ge 0 $, with the usual convention $\prod_{i=1}^{N_t} \tilde{Q}(s,t-T_i)'=I$ if $N_t=0$. Besides, it satisfies the following multidimensional integral equation:
\begin{equation}\label{ren_eq_psi_tilda}
\tilde{\psi}(s,t)=\bar{F}(t)I+ \int_0^t \tilde{Q}(s,t-y)'  \tilde{\psi}(s,t-y) dF(y),\quad \forall s\in\nbR^k,\ t\ge 0.
\end{equation}
\end{prop}
\begin{proof}
Decomposing according to $N_t=0$ and $N_t>0$ yields that
\begin{equation}\label{decompose_psi}
\tilde{\psi}(s,t)=\left[ \nbP(X_0=y|X_0=x)\nbP(N_t=0)\right]_{(x,y)\in {\cal S}^2} + \left[ \nbE\left( \left. e^{<s,\tilde{Z}(t)>}\nbu_{[X_{N_t}=y]} \nbu_{[N_t>0]}\right| X_0=x\right) \right]_{(x,y)\in {\cal S}^2}.
\end{equation}
Note that $\nbP(X_0=y|X_0=x)\nbP(N_t=0)=\delta_{x,y} \bar{F}(t)$, so that the first term on the right-hand side of \eqref{decompose_psi} is given by the the first term on the right-hand side of \eqref{eq_psi_tilda}. We turn to the second term on the right-hand side of \eqref{decompose_psi}.
Let us define ${\cal F}=\sigma (T_i,\ i\in \nbN)$ the sigma field generated by $T_i$'s for $i\in \nbN$ as well as the set of rvs
$$
s_t(x,y):= \nbE\left( \left. e^{<s,\tilde{Z}(t)>}\nbu_{[X_{N_t}=y]}\nbu_{[N_t>0]}\right| X_0=x,\ {\cal F}\right),\quad x,y \in {\cal S},
$$
where $s=(s_1,...,s_k)\in \nbR^k$ is fixed throughout the proof. Using $\tilde{Z}(t)=e^{\alpha t}Z(t;\alpha)$ and \eqref{def_Z_t}, one obtains
\begin{equation}\label{compute_s_t1}
s_t(x,y)=  \nbE\bigg( \nbu_{[N_t>0]}\prod_{i=1}^{N_t}\exp \bigg\{ \sum_{j=1}^k s_j X_{ij} e^{-\alpha (L_{ij}-(t-T_i))}\nbu_{[L_{ij}>t-T_i]}\bigg\}\nbu_{[X_{N_t}=y]}\bigg| X_0=x,\ {\cal F}\bigg).
\end{equation}
In order to compute $s_t(x,y)$, we use the fact that the Markov Chain $(X_i)_{i\in\nbN}$ is independent from ${\cal F}$ and $(L_{ij})_{i\in\nbN, j=1,..., k}$. Using the result in Corollary \ref{coro_compute_Mk} with replacement of $S_i:= X_{i-1}$, $f(S_1)=f(X_0)=1$ and $f_i(S_i)=f_i(X_{i-1}):=\exp \left\{ \sum_{j=1}^k s_j X_{i-1,j} e^{-\alpha (L_{(i-1),j}-(t-T_{i-1}))}\nbu_{[L_{i-1,j}>t-T_{i-1}]}\right\}$ for $i=2,...,l$ when $l=N_t+1$,  \eqref{compute_s_t1} may be expressed as
$$
\left[ s_t(x,y)\right]_{(x,y)\in {\cal S}^2}= \nbu_{[N_t>0]}\ I .\prod_{i=2}^{N_t+1} \tilde{Q}(s,t-T_{i-1})' =\nbu_{[N_t>0]}\prod_{i=1}^{N_t} \tilde{Q}(s,t-T_i)',
$$
where we recall that $I$ is the identity matrix and $\tilde{Q}(.,.)$ is defined in (\ref{def_pi_Q_tilda1}). Since $\nbE([s_t(x,y)]_{(x,y)\in {\cal S}^2} )$ is the second term in the right-hand side of \eqref{decompose_psi}, one thus obtains \eqref{eq_psi_tilda}.
%the second term in the right-hand side of \eqref{eq_psi_tilda} is given by $\nbE\left(\left[s_t(x,y)\right]_{(x,y)\in {\cal S}^2} \right) =\tilde{\psi}(s,t)$.

Finally, \eqref{ren_eq_psi_tilda} is obtained by considering again $[N_t=0]\iff [T_1>t]$ and $[N_t>0]\iff [T_1\le t]$ and conditioning with respect to $T_1$.
\end{proof}
It is known that a multidimensional integral equation such as \eqref{ren_eq_psi_tilda} cannot be solved in general. One particular case is when arrivals occur according to a Poisson process, in which case one has the following result.
\begin{prop}\label{prop_Poisson_psi}
If $\tau\sim {\cal E}(\lambda)$ for $\lambda>0$, then $\tilde{\psi}(s,t)$ is the unique solution to the first-order linear (matrix) differential equation
\begin{equation}\label{Poisson_ODE}
\partial_t \tilde{\psi}(s,t) = [-\lambda I+ \lambda\tilde{Q}(s,t)']\tilde{\psi}(s,t)=[\lambda (P-I) + \lambda P(\tilde{\pi}(s,t)-I)]\tilde{\psi}(s,t)
\end{equation}
with the initial condition $\tilde{\psi}(s,0)= I$ from (\ref{eq_psi_tilda}).
\end{prop}
Note that, even though \eqref{Poisson_ODE} admits a unique solution, $\tilde{\psi}(s,t)$ is not explicit except for particular cases e.g. when $P=I$. That is, when there is no  Markov interference in the model, the differential equation \eqref{Poisson_ODE} can be solved componentwise. Also note that a similar differential equation was obtained when $\alpha=0$ in \cite[Theorem 3.1]{MT02} for the joint generating function, when interarrival times are matrix exponentially distributed.
%that one can write in the form
%$$
%\tilde{\psi}(s,t)=\bar{F}(t) I + \int_0^t 
%$$
%By \cite[Lemma 2.1]{AM75},
\begin{proof}
Setting $dF(y)=\lambda e^{-\lambda y}dy$ and $\bar{F}(t)=e^{-\lambda t}$ in \eqref{ren_eq_psi_tilda} and differentiating with respect to $t$ yields \eqref{Poisson_ODE}.
\end{proof}

\subsection{The first and second moments}\label{sec:first_second_workload}
We are now interested in the first two moments defined in \eqref{def_moments}.  It is standard that $M_j(t)$ and $M_{jj'}(t)$ are linked to  $\tilde{\psi}(s,t)$ by
$$
M_j(t)=\left.\partial_{s_j}\tilde{\psi}(s,t)\right|_{s={\bf 0}},\quad M_{jj'}(t)=\left.\partial_{s_j}\partial_{s_{j'}}\tilde{\psi}(s,t)\right|_{s={\bf 0}},\quad j,j'=1,...,k .
$$
 It requires to differentiate $\tilde{\pi}(s,r)$ in (\ref{def_pi_Q_tilda}) with respect to $s_j$ or $s_j$ and $s_{j'}$ followed by putting $s={\bf 0}$. One obtains
\begin{eqnarray}
\left.\partial_{s_j}\tilde{\pi}(s,r)\right|_{s={\bf 0}}&=&\nbE \left( e^{-\alpha (L_j-r)}\nbu_{[L_j>r]}\right)\Delta_j,\label{derive1_pi_tilde}\\
\left.\partial_{s_{j'}}\partial_{s_j}\tilde{\pi}(s,r)\right|_{s={\bf 0}}&=& \nbE \left( e^{-\alpha (L_j-r)}\nbu_{[L_j>r]}e^{-\alpha (L_{j'}-r)}\nbu_{[L_{j'}>r]}\right)\Delta_j \Delta_{j'},\label{derive2_pi_tilde}
\end{eqnarray}
where $\Delta_j$ is given by (\ref{Di}).
Moreover, one also needs to compute $\tilde{\psi}({\bf 0},r)$, namely from \eqref{eq_psi_tilda}
\begin{equation}
\tilde{\psi}({\bf 0},r)=\nbE\bigg(  \prod_{i=1}^{N_r} \tilde{Q}({\bf 0},r-T_i)'\bigg)= \nbE\left( P^{N_r} \right),\quad r\ge 0 ,\label{psi_tilda_zero}
%&=& \bar{F}(r) I + \nbE\left( P^{N_r-1}\nbu_{[N_r \ge 1]}\right),
\end{equation}
as indeed $ \tilde{Q}({\bf 0},r)=P'$, again with the convention $P^0=I$. It is convenient to introduce the following notations for all $j,j'=1,...,k$:
\begin{eqnarray}
b_j(t)&:=&\int_0^t \Big[\left.\partial_{s_j}\tilde{\pi}(s,t-y)\right|_{s={\bf 0}}\Big] P \tilde{\psi}({\bf 0},t-y)dF(y),\label{bi1}\\
%&=& \int_0^t \nbE \left( e^{-\alpha (L_j-(t-y))}\nbu_{[L_j>t-y]}\right)\Delta_j PdF(y)\label{bi1}\\
b_{jj'}(t)&:=& \int_0^t \Big[\left.\partial_{s_{j'}}\partial_{s_j}\tilde{\pi}(s,t-y)\right|_{s={\bf 0}}\Big] P \tilde{\psi}({\bf 0},t-y)dF(y)\nonumber\\
&&+ \int_0^t \Big[\left.\partial_{s_j}\tilde{\pi}(s,t-y)\right|_{s={\bf 0}}\Big]P M_{j'}(t-y)dF(y)+ \int_0^t \Big[\left.\partial_{s_{j'}}\tilde{\pi}(s,t-y)\right|_{s={\bf 0}}\Big]P M_{j}(t-y)dF(y),\nonumber
\\
\label{bi2}
\end{eqnarray}
where $\partial_{s_j}\tilde{\pi}(s,t-y)|_{s={\bf 0}}$, $\partial_{s_{j'}}\partial_{s_j}\tilde{\pi}(s,t-y)|_{s={\bf 0}}$, $ \partial_{s_{j'}}\partial_{s_j}\tilde{\pi}(s,t-y)|_{s={\bf 0}} $ and $\tilde{\psi}({\bf 0},t-y)$ are given by \eqref{derive1_pi_tilde}, \eqref{derive2_pi_tilde} and \eqref{psi_tilda_zero}. Note that $b_j(t)$ is not explicit since \eqref{psi_tilda_zero} does not in general have a closed-form expression, and  so is $b_{jj'}(t)$. 
We may then obtain general results concerning $M_j(t)$ and $M_{jj'}(t)$. Following the notation in \cite[Section 2]{AM75}, we define, for a $N\times N$ dimensional matrix of non decreasing right continuous functions $(t\mapsto F_{ij}(t))_{i,j=1,...,N}$ and a $N\times N$ dimensional matrix of bounded measurable functions $(t\mapsto H_{ij}(t))_{i,j=1,...,N}$, the convolution $t\mapsto F\star H(t)= (F\star H)_{i,j=1,...,N}(t)$ by
$$
(F\star H)_{i,j}(t):= \sum_{h=1}^N \int_0^t H_{hj}(t-u)dF_{ih}(u)\quad i,j=1,...,N,\ t\ge 0.
$$
Then, differentiating \eqref{ren_eq_psi_tilda} with respect to $s_j$ or $s_j$ and $s_{j'}$ followed by putting $s={\bf 0}$ yields the following proposition. 
\begin{prop}\label{prop_renewal_equations_moments}
For $j,j'=1,...,k$, $M_j(t)$ and $M_{jj'}(t)$ satisfy the following multidimensional renewal equations
\begin{eqnarray}
M_j(t)&=& b_j(t) + (PF)\star M_j(t),\quad t\ge 0,\label{renewal_Mi}\\
M_{jj'}(t)&=& b_{jj'}(t) + (PF)\star M_{jj'}(t),\quad t\ge 0,\label{renewal_Mi2}
\end{eqnarray}
where $(PF)(t):=(p(x,y)F(t))_{(x,y)\in {\cal S}^2}$, and $b_j(t)$ and $b_{jj'}(t)$ are given by \eqref{bi1} and \eqref{bi2}.
\end{prop}
Although the solution for \eqref{ren_eq_psi_tilda} does not have a closed-form expression, it turns out that \eqref{bi1} and \eqref{bi2} have solutions which can be expressed in terms of a multidimensional renewal function. Indeed, since interarrival times satisfy $\tau>0$ a.s., one has that $(PF)(0)$ is the zero matrix, of which largest eigenvalue is thus $0$. \cite[Lemma 2.1]{AM75} entails that
\begin{equation}\label{sol_renewal}
M_j(t)=U\star b_j (t),\quad M_{jj'}(t)=U\star b_{jj'} (t),\quad  t\ge 0,
\end{equation}
where $U(t)$ is the renewal function defined by $U(t):=\sum_{n=0}^\infty (PF)^{\star (n)}(t)$, an ${\cal S}\times {\cal S}$ matrix, see \cite[Definition (2.3)]{AM75}. At this point, solutions in  \eqref{sol_renewal} are still not satisfactory because $U(t)$, $b_j (t)$, and $b_{jj'} (t)$ are not explicit. Hence, limiting behaviours of $M_j(t)$ and $M_{jj'}(t)$ are studied instead given as below.
\begin{lemm}\label{lemma_asymptotic_moment_general}
Let us suppose that $\tau$ is non lattice, then one has the following
\begin{eqnarray}
M_j(t) &\longrightarrow &\frac{1}{\nbE(\tau)}{\bf 1}\pi \int_0^\infty b_j(t) dt  ,\quad t\to\infty, \ j=1,...,k,\label{limit_moment1}\\
M_{jj'}(t) &\longrightarrow &\frac{1}{\nbE(\tau)}{\bf 1}\pi \int_0^\infty b_{jj'}(t) dt  ,\quad t\to\infty, \ j,j'=1,...,k.\label{limit_moment2}
\end{eqnarray}
\end{lemm}
\begin{proof}
Since $(PF)(\infty)=P$ has a spectral radius equal to $1$ and has (row vector) $\pi$ and ${\bf 1}$ (column vector) as left and right eigenvectors associated to the eigenvalue $1$, the renewal equation \eqref{renewal_Mi} satisfied by $M_j(t)$ and \cite[Theorem 2.2 (iii)]{AM75} yield \eqref{limit_moment1}. The same method applied to the renewal equation \eqref{renewal_Mi2} satisfied by $M_{jj'}(t)$ yields \eqref{limit_moment2}.
\end{proof}
Note that the limit in \eqref{limit_moment1} is not clearly available because an explicit expression for $\nbE(P^{N_r})$ is required to integrate $b_j(t)$ (see \eqref{bi1} with $\tilde{\psi}({\bf 0},t-y)$ given by \eqref{psi_tilda_zero}). Likewise, the limit in \eqref{limit_moment2} is not explicit either as it requires analytic expressions for $ M_j(t)$ and $M_{j'}(t)$ in \eqref{bi2}. However, the limiting first moment is explicitly available as below:
\begin{prop}\label{prop_asymptotic_moment_general_vector}
Assuming that $\tau$ is non lattice, the expectation $M_j(t){\bf 1}=[\nbE (\tilde{Z}_j(t)| \ X_1=x)]_{x\in {\cal S}} '$ asymptotically behaves as
\begin{equation}\label{limit_moment_simple}
M_j(t){\bf 1}\longrightarrow \frac{ \nbE(X_j)}{\nbE(\tau)}
\bigg[\frac{1-{\cal L}_j(\alpha)}{\alpha}\bigg]{\bf 1},\quad t\to\infty, \   j=1,...,k.%\pi\ \mbox{\normalfont diag} \left[ x_i,\ x=(x_1,...,x_k)\in{\cal S}\right],\quad t\to\infty, \ \forall i=1,...,k.
\end{equation}
\end{prop}
\begin{proof}
Let us prove that one obtains \eqref{limit_moment_simple} by post multiplying \eqref{limit_moment1} by ${\bf 1}$.  First note that, as $P {\bf 1}={\bf 1}$, $b_j(t){\bf 1}$ reduces thanks to \eqref{bi1}, \eqref{derive1_pi_tilde} and \eqref{psi_tilda_zero} to
\begin{eqnarray}
b_j(t){\bf 1}&=& \int_0^t \nbE \left( e^{-\alpha (L_j-(t-y))}\nbu_{[L_j>t-y]}\right)\Delta_j .P \nbE\left( P^{N_{t-y}}\right){\bf 1} dF(y)\nonumber\\
&=& \int_0^t \nbE \left( e^{-\alpha (L_j-(t-y))}\nbu_{[L_j>t-y]}\right)dF(y) .\ \Delta_j {\bf 1}\nonumber\\
&=& \nbE \left( e^{-\alpha (L_j-(t-\tau))}\nbu_{[0 \le t-\tau <  L_j]}\right)\Delta_j {\bf 1}.\label{proof_limit_Mi2}
\end{eqnarray}
In order to compute $\int_0^\infty b_j(t){\bf 1} dt$, one calculates
\begin{multline*}
\int_0^\infty \nbE \left( e^{-\alpha (L_j-(t-\tau))}\nbu_{[L_j>t-\tau\ge 0]}\right) dt=\int_0^\infty \nbE \left( e^{-\alpha (L_j-(t-\tau))}\nbu_{[\tau \le t < L_j+\tau]}\right) dt\\
=\nbE \left( \int_\tau^{L_j+\tau} e^{-\alpha (L_j-(t-\tau))} dt\right)=\nbE \left( \frac{1}{\alpha}\left[ 1- e^{-\alpha L_j}\right]\right)=\frac{1-{\cal L}_j(\alpha)}{\alpha},
\end{multline*}
so that one obtains from \eqref{proof_limit_Mi2} and \eqref{limit_moment1} that, as $t\to\infty$,
$$
M_j(t){\bf 1} \longrightarrow \frac{1}{\nbE(\tau)}{\bf 1}\pi \int_0^\infty b_j(t){\bf 1} dt = \frac{1}{\nbE(\tau)} \bigg[\frac{1-{\cal L}_j(\alpha)}{\alpha}\bigg] {\bf 1}\pi \ \Delta_j {\bf 1}.
$$
One checks easily that $ \pi \ \Delta_j\ {\bf 1}=\nbE(X_j)$  where $\Delta_j$ is given in \eqref{Di}, yielding \eqref{limit_moment_simple}.
\end{proof}

\subsection{The workload}\label{sec:workload}
%We now turn to the workload of the queues jointly to the state of $X_{N_t}$.
%\textcolor{blue}{and that at each time $T_i$ there is an arriving batch of customers of size $X_{i,j}$ within queue $j=1,...,k$, each customer within this batch with same service time $L_{ij}$.}\marge{Remove this?? already introduced} 
In this section, in the context of infinite server queues, following \cite[Section 5.2]{RW17}, we then consider the workload of the queues as the vector of $k$ processes $D(t)=(D_1(t),...,D_k(t))$ with each process defined as
\begin{equation*}
D_j(t):= \sum_{i=1}^{N_t}X_{ij}(T_i+L_{ij}-t)\nbu_{[t<T_i+L_{ij}]},\quad t\ge 0,\qquad j=1,...,k.
\end{equation*}
and that one has for all $t\ge 0$ and $j=1,...,k$
\[
D_j(t)=\left. -\frac{\partial}{\partial \alpha} \tilde{Z}_j(t;\alpha)\right|_{\alpha =0}.
\]
Here we define the joint expectation of the workload and the state of $X_{N_t}$ given the initial state of $X_0$ as
\begin{eqnarray}\label{Wit}
W_j(t)&:=& \left[\nbE \left(\left.D_j(t)\nbu_{[X_{N_t}=y]}\right|X_0=x\right)\right]_{(x,y)\in {\cal S}^2}\\
&=&\left. -\frac{\partial}{\partial \alpha} M_j(t)\right|_{\alpha =0}=\left[ -\frac{\partial}{\partial \alpha} \left[ \partial_{s_j}\tilde{\psi}(s,t;\alpha)\right]_{s={\bf 0}}\right]_{\alpha =0}.\nonumber
\end{eqnarray}
The following results the analogs of Proposition \ref{prop_renewal_equations_moments}, Lemma \ref{lemma_asymptotic_moment_general} and Proposition \ref{prop_asymptotic_moment_general_vector}.  First, let us define and compute 
\begin{equation}\label{di}
 \ell_j(t):=  \int_0^t \left[ -\frac{\partial}{\partial \alpha}\left[\partial_{s_j}\tilde{\pi}(s,t-y)\right]_{s={\bf 0}} \right]_{\alpha =0}P \tilde{\psi}({\bf 0},t-y)dF(y),
\end{equation}
where it follows from (\ref{derive1_pi_tilde}) that
\begin{equation}\label{di_extra}
\left[ -\frac{\partial}{\partial \alpha}\left[\partial_{s_j}\tilde{\pi}(s,r)\right]_{s={\bf 0}} \right]_{\alpha =0}= \nbE\left((L_j-r)\nbu_{[L_j>r]}\right)\ \Delta_j.
\end{equation} 
Consequently, the following proposition is provided.
\begin{prop}\label{prop_workload}
The joint expectation of the workload and the state of $X_{N_t}$ satisfies
\begin{equation}
W_j(t) = \ell_j(t) + (PF)\star W_j(t),\quad t\ge 0, \quad j=1,...,k,
\label{renewal_Di}
\end{equation}
of which asymptotic expression is given by
\begin{equation}
W_j(t) \longrightarrow \frac{1}{\nbE(\tau)}{\bf 1}\pi \int_0^\infty \ell_j(t) dt  ,\quad t\to\infty, \quad \ j=1,...,k.\label{limit_workload}
\end{equation}
Moreover, the asymptotic expected workload $W_j(t){\bf 1}=\left[\nbE \left(\left.D_j(t)\right|X_0=x\right)\right]_{x\in {\cal S}} '$ is given by
\begin{equation}\label{limit_workload_simple}
W_j(t){\bf 1}\longrightarrow \left[ \frac{\nbE(L_j^2)}{2\nbE(\tau)} +\frac{\nbE(\tau^2)}{2\nbE(\tau)}+\nbE(L_j)\right] \ \nbE(X_j){\bf 1},\quad t\to\infty, \ j=1,...,k. 
\end{equation}
\end{prop}
\begin{proof}
Since the proof of \eqref{renewal_Di} and \eqref{limit_workload} is analogous to the one of \eqref{renewal_Mi} and \eqref{limit_moment1}, our focus is on \eqref{limit_workload_simple}. As in the proof of Proposition \ref{prop_asymptotic_moment_general_vector} similar to \eqref{proof_limit_Mi2}, using $P{\bf 1}={\bf 1}$ one finds
$$
\ell_j(t){\bf 1}=\nbE\left( (L_j+\tau-t)\nbu_{[L_j+\tau>t]}\right)\Delta_j {\bf 1},
$$
with $\int_0^\infty \nbE\left( (L_j+\tau-t)\nbu_{[L_j+\tau>t]}\right) dt=\nbE\left( \int_0^{L_j+\tau} (L_j+\tau-t) dt\right) =\nbE\left( (L_j+\tau)^2/2\right)$. Using independence of $L_j$ and $\tau$ in the last expectation as well as $ \pi  \Delta_j {\bf 1}=\nbE(X_j)$ yields \eqref{limit_workload_simple} by post multiplying \eqref{limit_workload} by ${\bf 1}$.
\end{proof}

\section{Special cases}\label{sec:particular_cases}

The results given in Proposition \ref{prop_psi_tilda}, Lemma \ref{lemma_asymptotic_moment_general} and Proposition \ref{prop_workload} hold under general assumptions on the service times ($L_j$) and interarrival times ($\tau$). We present here some particular cases where those results are more explicitly obtainable with a specific distributional assumption for $L_j$ or $\tau$. Namely, as is customary when studying infinite server queues, one expects reasonably to obtain more information when one of those two rvs are exponentially distributed, see \cite[Chapter 3]{T62}.

\subsection{Exponentially distributed service times}\label{subsec:expo_service_times}
First, it is assumed that service time $L_j$ for $j$-type customer is exponentially distributed with parameter $\mu_j>0$. To provide explicit expressions for the limits of $M_j(t)$, $M_{jj'(t)}$ and $W_j(t)$ as $t\to\infty$ in \eqref{limit_moment1}, \eqref{limit_moment2} and \eqref{limit_workload}, we define the Laplace transforms for $\tilde{\psi}(s,t)$, $M_j(t)$, and $b_j(t)$ by
\begin{eqnarray}
\hat{\psi}(s,h)&=&\int_0^\infty \tilde{\psi}(s,t) e^{-ht}dt,\quad s\in\nbR^k ,\nonumber\\
\hat{M}_j(h) &=& \int_0^\infty M_j(t) e^{-ht}dt,\nonumber\\
\hat{b}_j(h) &=& \int_0^\infty b_j(t) e^{-ht}dt,\label{hatbi}
\end{eqnarray}
respectively for all $h>0$. Next, some relations between the above quantities are first given.
\begin{prop}
The Laplace transforms verify for all $h>0$
\begin{eqnarray}
\hat{\psi}({\bf 0},h)&=&\frac{1- {\cal L}^\tau(h)}{h} \left( I-{\cal L}^\tau(h)P\right)^{-1},\label{expression_psi_hat}\\
\hat{M}_j(h) &=& \left( I-{\cal L}^\tau(h)P\right)^{-1} \hat{b}_j(h),\quad j=1,...,k.\label{expression_Mi_hat}
\end{eqnarray}
\end{prop}
\begin{proof}
Recalling that $ \tilde{\pi}({\bf 0},r)=I$ from (\ref{def_pi_Q_tilda}), \eqref{ren_eq_psi_tilda} with $s={\bf 0}$ becomes the renewal equation
$$
 \tilde{\psi}({\bf 0},t)=\bar{F}(t) I + (PF)\star  \tilde{\psi}({\bf 0},.) (t),
$$
which, upon taking Laplace transforms on both sides, yields
$$
\hat{\psi}({\bf 0},h) = \frac{1- {\cal L}^\tau(h)}{h}I + P {\cal L}^\tau(h)\hat{\psi}({\bf 0},h).
$$
Then \eqref{expression_psi_hat} is obtained by noting that, since ${\cal L}^\tau(h)<1$ and $P$ is a stochastic matrix, the matrix $ {\cal L}^\tau(h)P$ has spectral radius less than $1$ hence $I-{\cal L}^\tau(h)P$ is invertible. Similarly, \eqref{expression_Mi_hat} is obtained by taking Laplace transforms in the renewal equation \eqref{renewal_Mi}.
\end{proof}
\begin{theorem}\label{prop_Mi_expo_service}
The asymptotic result for the first moment jointly to the state of $X_{N_t}$ in \eqref{limit_moment1} can be precisely expressed as
 \begin{equation}\label{expression_Mi_expo_service}
M_j(t) \longrightarrow \frac{1}{\nbE(\tau)} 
\bigg[\frac{1- {\cal L}^\tau(\mu_j)}{\mu_j+\alpha}\bigg]{\bf 1}\pi \Delta_j P  \left( I-{\cal L}^\tau(\mu_j)P\right)^{-1},\quad t\to\infty, \   j=1,...,k.
\end{equation}
\end{theorem}
\begin{proof}
When $L_j\sim {\cal E}(\mu_j)$, one computes 
\begin{equation}\label{T10a}
\nbE \left( e^{-\alpha (L_j-r)}\nbu_{[L_j>r]}\right)=
\int^\infty_r e^{-\alpha(t-r)}\mu_j e^{-\mu_j t}dt=\mu_j e^{\alpha r} \int^\infty_r e^{-(\mu_j+\alpha)t}dt
=\frac{\mu_j}{\mu_j+\alpha} e^{-\mu_j  r}, \quad r\ge 0,
\end{equation}
so that one has \eqref{bi1} from \eqref{derive1_pi_tilde} and \eqref{T10a} that
\begin{eqnarray}
b_j(t)&=&\frac{\mu_j}{\mu_j+\alpha}\ \Delta_j P \int_0^t e^{-\mu_j (t-y)}\tilde{\psi}({\bf 0},t-y)dF(y)\nonumber\\
&=& \frac{\mu_j}{\mu_j+\alpha}\ \Delta_j P \ \nbE\left[e^{-\mu_j (t-\tau)}\tilde{\psi}({\bf 0},t-\tau)\nbu_{[t\ge \tau]}\right]. \label{expression_useful_bi}
\end{eqnarray}
The right-hand side of \eqref{limit_moment1} is thus computed as
\begin{eqnarray*}
&&\frac{1}{\nbE(\tau)}{\bf 1}\pi \int_0^\infty b_j(t) dt\\
&=& \frac{1}{\nbE(\tau)}{\bf 1}\pi \frac{\mu_j}{\mu_j+\alpha}\ \Delta_j P \ \int_0^\infty \nbE\left[e^{-\mu_j (t-\tau)}\tilde{\psi}({\bf 0},t-\tau)\nbu_{[t\ge \tau]}\right] dt\\
&=& \frac{1}{\nbE(\tau)}{\bf 1}\pi \frac{\mu_j}{\mu_j+\alpha}\ \Delta_j P \ \nbE\left[ \int_\tau^\infty e^{-\mu_j (t-\tau)}\tilde{\psi}({\bf 0},t-\tau)dt \right]\\
&=& \frac{1}{\nbE(\tau)}{\bf 1}\pi \frac{\mu_j}{\mu_j+\alpha}\ \Delta_j P \ \hat{\psi}({\bf 0},\mu_j),
\end{eqnarray*}
and in turn, \eqref{expression_Mi_expo_service} is obtained thanks to \eqref{expression_psi_hat}.
\end{proof}
Let us note that the previous proof enables us to similarly obtain the expression of $\hat{b}_j(h)$ defined in \eqref{hatbi} thanks to \eqref{expression_useful_bi} as follows
\begin{eqnarray}
\hat{b}_j(h)&=& \frac{\mu_j}{\mu_j+\alpha}\ \Delta_j P \ \int_0^\infty e^{-ht} \nbE\left[e^{-\mu_j (t-\tau)}\tilde{\psi}({\bf 0},t-\tau)\nbu_{[t\ge \tau]}\right] dt\nonumber\\
&=& \frac{\mu_j}{\mu_j+\alpha}\ \Delta_j P \  \nbE\left[\int_\tau^\infty e^{-h(t-\tau)}e^{-\mu_j (t-\tau)}\tilde{\psi}({\bf 0},t-\tau) dt\ .e^{-h\tau}\right] \nonumber\\
&=& \frac{\mu_j}{\mu_j+\alpha}\ \Delta_j P \  \hat{\psi}({\bf 0},\mu_j+h).{\cal L}^\tau (h),\quad h>0 .\label{expression_useful_bi_hat} 
\end{eqnarray}
\begin{theorem}\label{theo_second_limit_moment_expo}
The asymptotic result for the second moment jointly to the state of $X_{N_t}$ in \eqref{limit_moment2} can be precisely expressed as
\begin{multline}\label{expression_Mii_expo_service}
M_{jj'}(t) \longrightarrow \frac{1}{\nbE(\tau)}\bigg(\frac{\mu_j}{\mu_j+\alpha}\bigg)\bigg( \frac{\mu_{j'}}{\mu_{j'}+\alpha} \bigg)\bigg[\frac{1-{\cal L}^\tau (\mu_j+\mu_{j'})}{\mu_j+\mu_{j'}}\bigg]
{\bf 1}\pi\Big\{ \Delta_j \Delta_{j'}+ {\cal L}^\tau (\mu_j)\  \Delta_j P \left( I-{\cal L}^\tau(\mu_j)P\right)^{-1} \Delta_{j'} \\
 + {\cal L}^\tau (\mu_{j'}) \ \Delta_{j'} P \left( I-{\cal L}^\tau(\mu_{j'})P\right)^{-1} \Delta_j \Big\}  P \  \left( I-{\cal L}^\tau(\mu_{j'}+\mu_j)P\right)^{-1}
\end{multline}
as $t\to \infty$, when $j,j'=1,...,k$, $j\neq j'$, and
\begin{multline}\label{expression_Mi_expo_servicebis}
M_{jj'}(t)\longrightarrow \frac{1}{\nbE(\tau)} \bigg[ \frac{1-{\cal L}^\tau (\mu_j )}{\mu_j+2\alpha} \bigg]{\bf 1}\pi\Delta_j^2 P \ \left( I-{\cal L}^\tau(\mu_j)P\right)^{-1} \\
+ \frac{2{\cal L}^\tau (\mu_j)}{\nbE(\tau)} \bigg[  \frac{1- {\cal L}^\tau (2\mu_j)}{2\mu_j}\bigg]
\left(\frac{\mu_j}{\mu_j+\alpha} \right)^2{\bf 1}\pi\Delta_j P
\left( I-{\cal L}^\tau(\mu_j)P\right)^{-1}\Delta_j P \  \left( I-{\cal L}^\tau(2\mu_j)P\right)^{-1}
\end{multline}
as $t\to \infty$, when $j,j'=1,...,k$, $j=j'$. We remark that \eqref{expression_Mii_expo_service} and \eqref{expression_Mi_expo_servicebis} still hold when $\mu_j$ or $\mu_{j'}$ is infinite, i.e. when the corresponding delays $L_j$ or $L_{j'}$ are $0$.
\end{theorem}
\begin{proof}
One first computes that 
\begin{equation}\label{expectation_expo_service_Mi2}
\nbE \left( e^{-\alpha (L_j-r)}\nbu_{[L_j>r]}e^{-\alpha (L_{j'}-r)}\nbu_{[L_{j'}>r]}\right)= 
\left\{
\begin{array}{ll}
\frac{\mu_j}{\mu_j+\alpha} \frac{\mu_{j'}}{\mu_{j'}+\alpha} e^{-\mu_j  r} e^{-\mu_{j'} r} & \quad \mbox{if } j\neq j',\\
\frac{\mu_j}{\mu_j+2\alpha} e^{-\mu_j r}& \quad \mbox{if } j=j'
\end{array}
\right.
\end{equation}
for $r\ge 0$. To evaluate the integral in \eqref{limit_moment2} with \eqref{bi2}, we thus need to compute the following integrals:
\begin{eqnarray}
&& \int_0^\infty \int_0^t  \left.\partial_{s_{j'}}\partial_{s_j}\tilde{\pi}(s,t-y)\right|_{s={\bf 0}} P \tilde{\psi}({\bf 0},t-y)dF(y)\label{Mi2_expo_service1},\\
&& \int_0^\infty \int_0^t \left.\partial_{s_j}\tilde{\pi}(s,t-y)\right|_{s={\bf 0}}P M_{j'}(t-y)dF(y),\label{Mi2_expo_service2}
\end{eqnarray}
for $j,j'=1,...,k$. When $j\neq j'$, using (\ref{derive2_pi_tilde}) with \eqref{expectation_expo_service_Mi2} followed by applying \eqref{expression_psi_hat}, \eqref{Mi2_expo_service1} may be expressed as
\begin{eqnarray}\label{expression36}
&&\bigg(\frac{\mu_j}{\mu_j+\alpha} \bigg)\bigg(\frac{\mu_{j'}}{\mu_{j'}+\alpha}\bigg)\Delta_j\Delta_{j'}P\int_0^\infty \nbE\left[ e^{-(\mu_j +\mu_{j'}) (t-\tau)} \tilde{\psi}({\bf 0},t-\tau)\nbu_{[t\ge \tau]}\right]dt\nonumber\\
&&= \bigg(\frac{\mu_j}{\mu_j+\alpha} \bigg)\bigg(\frac{\mu_{j'}}{\mu_{j'}+\alpha}\bigg)\Delta_j\Delta_{j'}P \ \hat{\psi}({\bf 0},\mu_j+\mu_{j'})\nonumber\\
&&= \bigg(\frac{\mu_j}{\mu_j+\alpha}\bigg)\bigg( \frac{\mu_{j'}}{\mu_{j'}+\alpha} \bigg)\bigg[\frac{1-{\cal L}^\tau (\mu_j+\mu_{j'})}{\mu_j+\mu_{j'}}\bigg] \Delta_j\Delta_{j'}P \ \left( I-{\cal L}^\tau(\mu_j+\mu_{j'})P\right)^{-1} .\nonumber\\
\end{eqnarray}
When $j=j'$, similar computation yields that \eqref{Mi2_expo_service1} is expressed as
\begin{equation}\label{eqi}
\bigg[ \frac{1-{\cal L}^\tau (\mu_j)}{\mu_j+2\alpha}\bigg] \Delta_j^2 P \ \left( I-{\cal L}^\tau(\mu_j)P\right)^{-1},
\end{equation}
where $\Delta_j^2=\mathrm{diag} [ x_j^2,\ x=(x_1,...,x_k)\in{\cal S}]$ for $j=1,...,k$. Turning to \eqref{Mi2_expo_service2}, replacing \eqref{derive1_pi_tilde} with (\ref{T10a}) followed by using \eqref{expression_Mi_hat} and \eqref{expression_useful_bi_hat} with \eqref{expression_psi_hat} yields
\begin{eqnarray}
&& \frac{\mu_j}{\mu_j+\alpha} \Delta_j P\ \int_0^\infty \nbE \left[ e^{-\mu_j(t-\tau)} M_{j'}(t-\tau) \nbu_{[t\ge \tau]}\right]dt= \frac{\mu_j}{\mu_j+\alpha} \Delta_j P\ \hat{M}_{j'}(\mu_j)\nonumber\\
&=& \frac{\mu_j}{\mu_j+\alpha} \Delta_j P\ \left( I-{\cal L}^\tau(\mu_j)P\right)^{-1} \hat{b}_{j'}(\mu_j)\nonumber\\
&=& \bigg(\frac{\mu_j}{\mu_j+\alpha}\bigg)\bigg( \frac{\mu_{j'}}{\mu_{j'}+\alpha}\bigg){\cal L}^\tau (\mu_j)\Delta_j P \left( I-{\cal L}^\tau(\mu_j)P\right)^{-1} \ \Delta_{j'} P \  \hat{\psi}({\bf 0},\mu_{j'}+\mu_j)\nonumber\\
&=& \bigg(\frac{\mu_j}{\mu_j+\alpha}\bigg)\bigg( \frac{\mu_{j'}}{\mu_{j'}+\alpha}
\bigg)\bigg[\frac{1- {\cal L}^\tau (\mu_{j'}+\mu_j)}{\mu_{j'}+\mu_j}\bigg]{\cal L}^\tau (\mu_j) 
\Delta_j P\left( I-{\cal L}^\tau(\mu_j)P\right)^{-1} \ \Delta_{j'} P \  \left( I-{\cal L}^\tau(\mu_{j'}+\mu_j)P\right)^{-1} .\nonumber\\
\label{expression37}
\end{eqnarray}
Then, gathering expressions \eqref{expression36} and \eqref{expression37} for \eqref{Mi2_expo_service1} and \eqref{Mi2_expo_service2} respectively yields \eqref{expression_Mii_expo_service}. Also, (\ref{expression_Mi_expo_servicebis}) is obtained with the help of (\ref{eqi}) and (\ref{expression37}).
\end{proof}
\begin{theorem}
The asymptotic result for the expectation of the workload jointly to the state of $X_{N_t}$ in \eqref{limit_workload} can be precisely expressed as
\begin{equation}\label{expression_Wi_expo_service}
W_j(t) \longrightarrow \frac{1}{\mu_j^2} \bigg[\frac{1- {\cal L}^\tau(\mu_j)}{\nbE(\tau)}\bigg]{\bf 1}\pi\Delta_j P\;  \left( I-{\cal L}^\tau(\mu_j)P\right)^{-1},\quad t\to\infty, \   j=1,...,k.
\end{equation}
\end{theorem}
\begin{proof}
When $L_j\sim {\cal E}(\mu_j)$, one straightforward verifies that $\nbE((L_j-r)\nbu_{[L_j>r]})=e^{-\mu_j r}/\mu_j$. Then, one has from \eqref{di} and \eqref{di_extra} that 
$$\ell_j(t)=\frac{1}{\mu_j} \ \Delta_j\ P\ \nbE\left(e^{-\mu_j (t-\tau)}  \tilde{\psi}({\bf 0},t-\tau)\nbu_{[t\ge \tau ]}\right),$$
from which the computation of $\frac{1}{\nbE(\tau)} {\bf 1}\pi \int_0^\infty \ell_j(t)dt$ in (\ref{limit_workload}) is led similarly to that of $\frac{1}{\nbE(\tau)} {\bf 1}\pi \int_0^\infty b_j(t)dt$ in Theorem \ref{prop_Mi_expo_service}. Hence, the result \eqref{expression_Wi_expo_service} follows by using (\ref{expression_psi_hat}).
\end{proof}

\begin{example}\normalfont
This example illustrates numerically convergences of \eqref{expression_Mi_expo_service} and \eqref{expression_Mi_expo_servicebis} for the first and second joint moments when $\alpha=0.1$. This was done by simulating $(Z(t),X_{N_t})$ in \eqref{def_Z_t} for (small) $t=30$ and (large) $t=100$ and estimating the left-hand side of \eqref{expression_Mi_expo_service} and \eqref{expression_Mi_expo_servicebis} thanks to the Law of Large Numbers (Monte Carlo) through $n=500$ iterations. The right-hand side of \eqref{expression_Mi_expo_service} and \eqref{expression_Mi_expo_servicebis} were computed explicitly by considering for the interarrival $\tau$ a Gamma distribution with shape $a$ and rate $b$ with the LT ${\cal L}^\tau(u)=\frac{1}{(1+u/b)^a}$ and $\nbE(\tau)=\frac{a}{b}$. We consider two cases of $(a,b)$ choosing $(1,10)$ (i.e. $\tau\sim {\cal E}(10)$), so that $\nbE(\tau)=0.1$, and $(0.75, 15)$ with $\nbE(\tau)=0.05$. 
Suppose that $k=1$, i.e. a one dimensional process $\{Z(t),\ t\ge 0\}$ and a Markov Chain $(X_i)_{i\in\nbN}$ with state space $\{0,1\}$ (i.e. $K=1$).  Assume that the transition matrix is given by $P=\left( \begin{array}{cc}
0.25 & 0.75\\ 0.5 & 0.5
\end{array}\right)$, with stationary distribution $\pi=(0.4, 0.6)$. All simulations and computations were done with {\tt Scilab}. We finally suppose that all delays have same distribution $L\sim {\cal E}(1)$.
\begin{table}[!h]
\begin{center}
\begin{tabular}{|c|c|c|}
\hline
 & $(a,b)=(1,10)$ & $(a,b)=(0.75,15)$\\
 \hline
Exact % expression in RHS of \eqref{expression_Mi_expo_service} 
& 
$\left( \begin{array}{cc}
   2.222  & 3.232\\
   2.222  & 3.232
\end{array}\right)$
& 
$\left( \begin{array}{cc}
   4.405  & 6.504\\
   4.405 &  6.504
\end{array}\right)$
\\
\hline
Monte Carlo for $t=30$
%estimate for $M_1(t)$ 
&
$\left( \begin{array}{cc}
  1.758 &  3.152\\
  2.218 &  3.566
\end{array}\right)$
& 
$\left( \begin{array}{cc}
   4.185  & 6.645\\
   3.925  & 6.503
\end{array}\right)$
\\
\hline
Monte Carlo for $t=100$
%estimate for $M_1(t)$ 
&
$\left( \begin{array}{cc}
2.242  & 3.153\\
1.892  & 3.574
\end{array}\right)$
& 
$\left( \begin{array}{cc}
  4.476 &  6.344\\
   4.411  & 6.509  
\end{array}\right)$
\\
\hline
\end{tabular}
\end{center}
\caption{First-order joint moments $M_1(t)$}
\end{table}

\begin{table}[!h]
\begin{center}
\begin{tabular}{|c|c|c|}
\hline
 & $(a,b)=(1,10)$ & $(a,b)=(0.75,15)$\\
 \hline
Exact 
%expression in RHS of \eqref{expression_Mi_expo_servicebis} 
& 
$\left( \begin{array}{cc}
  14.196 &  20.188\\
   14.196 &  20.188
\end{array}\right)$
& 
$\left( \begin{array}{cc}
   52.596  & 76.623\\
   52.596  & 76.623
\end{array}\right)$
\\
\hline
Monte Carlo for $t=30$
%estimate for $M_{11}(t)$ 
&
$\left( \begin{array}{cc}
   11.017  & 23.043\\
   12.322  & 22.652
\end{array}\right)$
& 
$\left( \begin{array}{cc}
   48.302  & 78.761\\
   47.395  & 82.891
\end{array}\right)$
\\
\hline
Monte Carlo for $t=100$
%estimate for $M_{11}(t)$ 
&
$\left( \begin{array}{cc}
   12.543 &  21.094\\
   14.192 &  22.089 
\end{array}\right)$
& 
$\left( \begin{array}{cc}
  53.635 &  77.217\\
  45.289 &  83.454
\end{array}\right)$
\\
\hline
\end{tabular}
\end{center}
\caption{Second-order joint moments $M_{11}(t)$}
\end{table}
Both tables above illustrate the convergence of the first-order and the second-order joint moments to the values calculated in \eqref{expression_Mi_expo_service} and \eqref{expression_Mi_expo_servicebis} respectively. Also note that this simple example gives us some idea of potential applications of the model. First, from a queueing point of view, the numerical model described here explains some infinite server queue system where, if an arriving customer is not admitted in the queue at time $T_i$ (i.e. $X_i=0$) then the next one arriving at time $T_{i+1}$ is accepted with high probability $0.75$; this is especially interesting in a congestion regulation context, where one may choose to accept incoming customers more easily when the previous ones were rejected with high probability. Furthermore, it would be interesting to utilize the model for an insurance company facing a situation where claims occurring at time $T_i$ are either immediately reported when $X_i=0$ or reported with delay $L_i$ when $X_i=1$. From $P$, some interesting feature appears as if a claim is not reported (resp. reported) at time $T_i$ then the next one is reported (resp. not reported) at time $T_{i+1}$ with probability $0.5$ (resp. with probability $0.75$). This model could reflect the policyholder's certain type of behaviour, e.g.  after immediately reporting a claim at time $T_i$, the policyholder prefers to delay reporting of the next claim at time $L_{i+1}+T_{i+1}$ with probability $0.75$ to avoid the increase of premium when the policyholder renews the insurance. On the other hand, it can also explain the opposite situation with a different transition matrix $P$. For example, a worker understands that if the reporting delay of workplace injury is longer, then it is harder to prove the injury is work related under the workers compensation claim. In this case, the transition probability from $X_{i}=0$ to $X_{i+1}=1$ is much lower, that is, the policyholder prefers to report the claim immediately.  
\end{example}

\subsection{Exponentially distributed interarrival times}
We now suppose in this subsection that $\tau\sim{\cal E}(\lambda)$, i.e. that arrivals occur according to a Poisson process with intensity $\lambda >0$. From Proposition \ref{prop_Poisson_psi}, it has been shown that the transient mgf $\tilde{\psi}(s,t)$ is the unique solution to an ordinary differential equation. Under this Poisson arrival setting, we shall derive closed-form expressions for the the transient behavior of the first and second moments as well as the expectation of the workload. To begin, a closed-form expression for $b_j(t)$ in \eqref{bi1} is obtained. In this case,
%Note that, as $N_r\sim {\cal P}(\lambda r)$ for all $r>0$, 
one finds that \eqref{psi_tilda_zero} becomes
\begin{equation}\label{tpsi0r}
\tilde{\psi}({\bf 0},r)=\nbE\left( P^{N_r} \right)=  e^{\lambda r (P-I)},\quad r>0,
\end{equation}
whence $b_j(t)$ in (\ref{bi1}) with \eqref{derive1_pi_tilde} may be expressed as 
\begin{eqnarray*}
b_j(t)&= &\Delta_j \int_0^t \nbE \left( e^{-\alpha (L_j-(t-y))}\nbu_{[L_j>t-y]}\right)  P e^{\lambda (t-y) (P-I)}   \lambda e^{-\lambda y}dy\nonumber\\
&=& \lambda\Delta_j e^{-\lambda t}\int_0^t \nbE \left( e^{-\alpha (L_j-y)}\nbu_{[L_j>y]}\right) P e^{\lambda yP}  dy .
\end{eqnarray*}
Furthermore, one checks easily that for all $t\ge 0$
\begin{equation}\label{bi_Poisson}
b_j'(t)+\lambda b_j(t)= \lambda \nbE \left( e^{-\alpha (L_j-t)}\nbu_{[L_j>t]}\right) \Delta_j P e^{\lambda t (P-I)}.
\end{equation}
\begin{theorem}\label{prop_exact_expression_Mi_Poisson}
One has the exact expression for the first joint moment given by
\begin{equation}\label{exact_expression_Mi_Poisson}
M_j(t)=\lambda e^{\lambda t (P-I)} \int_0^t \nbE \left( e^{-\alpha (L_j-v)}\nbu_{[L_j>v]}\right) e^{-\lambda v (P-I)}  \Delta_j P e^{\lambda v (P-I)}   dv,\quad t\ge 0,~j=1,...,k.
\end{equation}
\end{theorem}
\begin{proof}
We aim at obtaining a differential equation satisfied by $M_j(t)$. Remember from \eqref{renewal_Mi} that it satisfies the renewal matrix equation, with $dF(y)=\lambda e^{-\lambda y}dy$ as well as a change of variable $t-y:=y$,
$$
M_j(t)=b_j(t) + P \int_0^t M_j(t-y) \lambda e^{-\lambda y}dy = b_j(t) + e^{-\lambda t} P \int_0^t M_j(y) \lambda e^{\lambda y}dy,
$$
which, upon differentiation, and thanks to \eqref{bi_Poisson}, leads to the first-order matrix differential equation
\begin{eqnarray}\label{EDO_Mi_Poisson}
M_j'(t) &=& b_j'(t)   +\lambda  b_j(t) - \lambda M_j(t) + \lambda P M_j(t)\nonumber\\
&=& \lambda \nbE \left( e^{-\alpha (L_j-t)}\nbu_{[L_j>t]}\right) \Delta_j P e^{\lambda t (P-I)} + \lambda (P- I)M_j(t)
\end{eqnarray}
with initial condition $M_j(0)=0$. The solution to \eqref{EDO_Mi_Poisson} is given by \eqref{exact_expression_Mi_Poisson}.
\end{proof}
%\textcolor{red}{??The expression of $M_j(t)$ obtained in the previous proposition thus enables us to write from \eqref{bi2} that}

Next, using \eqref{tpsi0r}, \eqref{bi2} with \eqref{derive1_pi_tilde} and \eqref{derive2_pi_tilde} in this case is given by
\begin{eqnarray*}
b_{jj'}(t)&=&  \Delta_j\Delta_{j'} \int_0^t \nbE \left( e^{-\alpha (L_j-(t-y))}\nbu_{[L_j>t-y]}e^{-\alpha (L_{j'}-(t-y))}\nbu_{[L_{j'}>t-y]}\right)P e^{\lambda (t-y) (P-I)} \lambda e^{-\lambda y}dy\nonumber\\
&&+ \Delta_j \int_0^t \nbE \left( e^{-\alpha (L_j-(t-y))}\nbu_{[L_j>t-y]}\right)P M_{j'}(t-y)\lambda e^{-\lambda y}dy\\
&&+ \Delta_{j'} \int_0^t \nbE \left( e^{-\alpha (L_{j'}-(t-y))}\nbu_{[L_{j'}>t-y]}\right)P M_{j}(t-y)\lambda e^{-\lambda y}dy,
\end{eqnarray*}
and thus one finds the following relation
\begin{multline}\label{bi2_Poisson}
b_{jj'}'(t)+\lambda b_{jj'}(t)= \lambda \nbE \left( e^{-\alpha (L_j-t)}\nbu_{[L_j>t]}e^{-\alpha (L_{j'}-t)}\nbu_{[L_{j'}>t]}\right) \Delta_j\Delta_{j'}
P e^{\lambda t (P-I)}\\
+\lambda \nbE \left( e^{-\alpha (L_j-t)}\nbu_{[L_j>t]}\right) \Delta_j P M_{j'}(t) + \lambda \nbE \left( e^{-\alpha (L_{j'}-t)}\nbu_{[L_{j'}>t]}\right)\Delta_{j'} P M_{j}(t),
\end{multline}
where $M_j(t)$ is given in \eqref{exact_expression_Mi_Poisson}.
\begin{theorem}\label{theo_second_moment_Poisson}
One has the exact expression for the second moment given by
\begin{multline}\label{exact_expression_Mii_Poisson}
M_{jj'}(t)= \lambda e^{\lambda t (P-I)} \int_0^t e^{-\lambda v (P-I)}
\left\{ \nbE \left( e^{-\alpha (L_j-v)}\nbu_{[L_j>v]}e^{-\alpha (L_{j'}-v)}\nbu_{[L_{j'}>v]}\right)
\Delta_j\Delta_{j'} Pe^{\lambda v (P-I)} \right.\\
\left. +  \nbE \left( e^{-\alpha (L_j-v))}\nbu_{[L_j>v]}\right) \Delta_j P M_{j'}(v)
 +  \nbE \left( e^{-\alpha (L_{j'}-v))}\nbu_{[L_{j'}>v]}\right)\Delta_{j'} P M_{j}(v)\right\} dv,\quad t\ge 0,
\end{multline}
for $j,j'=1,...,k$, where $M_{j}(v)$ and $M_{j'}(v)$ are given by \eqref{exact_expression_Mi_Poisson} in Theorem \ref{prop_exact_expression_Mi_Poisson}.
\end{theorem}
Let us note that the structure of the expression of $M_{jj'}(t)$ is different according to whether $j=j'$ or $j\neq j'$, as $\nbE  ( e^{-\alpha (L_j-v)}\nbu_{[L_j>v]}e^{-\alpha (L_{j'}-v)}\nbu_{[L_{j'}>v]})$ is equal to $\nbE ( e^{-2\alpha (L_j-v)}\nbu_{[L_j>v]})$ if $j=j'$, or $\nbE  ( e^{-\alpha (L_{j}-v)}\nbu_{[L_j>v]}  ) \nbE  (e^{-\alpha (L_{j'}-v)}\nbu_{[L_{j'}>v]} )$ when $j\neq j'$, by independence.
\begin{proof}
Similar to the proof of Theorem \ref{prop_exact_expression_Mi_Poisson}, we write the renewal equation \eqref{renewal_Mi2} satisfied by $M_{jj'}(t)$ as
$M_{jj'}(t)=b_{jj'}(t) + P\int_0^t M_{jj'}(t-y)\lambda e^{-\lambda y}dy$, $t\ge 0$. The same first-order differential equation analysis with the expression of $b_{jj'}'(t)+\lambda b_{jj'}(t)$ given in \eqref{bi2_Poisson} yields thus the explicit expression \eqref{exact_expression_Mii_Poisson}.
\end{proof}
Finally, a transient expression for the expectation of the workload has the same structure as the first moment, and the following result is given without proof:
\begin{theorem}
One has the exact expression for the expectation of the workload given by
$$
W_j(t)=\lambda e^{\lambda t (P-I)} \int_0^t \nbE \left( (L_j-v)\nbu_{[L_j>v]}\right)e^{-\lambda v (P-I)}  \Delta_j  P e^{\lambda v (P-I)} dv,\quad t\ge 0,\ j=1,...,k.
$$
\end{theorem}

\section{Moment generating function for deterministic interarrival times}\label{sec:deterministic_interarriva}
So far, it has been shown that the (transient or limiting) distribution of process $\tilde{Z}(t)$ is hard to study explicitly in general, except for the Poisson arrivals. Hence we shall consider a specific distribution for the interarrival times being deterministic, and equal to $1$ without loss of generality to obtain some results on the mgf $\tilde{\psi}(s,t)$.
\begin{theorem}\label{theo_deterministic_psi}
Suppose that $\tau=1$ a.s., then $(\tilde{Z}(t)=\tilde{Z}(t;\alpha), X_{N_t})$ has a closed-form expression for the mgf given by
\begin{equation}\label{transient_deterministic_arrival}
\tilde{\psi}(s,t)=\tilde{\psi}(s,t;\alpha)=\bigg[ \prod_{m=0}^{t-1} \tilde{Q}(s,m)\bigg]'=  \prod_{m=1}^t \tilde{Q}(s,t-m)',\quad t\in \nbN ,
\end{equation}
where $\tilde{\pi}(s,t)$ and $\tilde{Q}(s,t)$ are given in \eqref{def_pi_Q_tilda} and (\ref{def_pi_Q_tilda1}) respectively. Besides, when $\nbE(L_j)$ is finite for all $j=1,...,k$, $\lim_{t\to \infty} \prod_{m=0}^{t} \tilde{Q}(s,m)=\prod_{m=0}^{\infty} \tilde{Q}(s,m)$ exists, and $(\tilde{Z}(t),X_{N_t})$ converges in distribution as $t\to \infty$ given $X_0=x$ towards $({\cal Z}_\infty,{\cal X}^x_\infty)\in\nbR^k\times {\cal S}$ with the joint mgf given by
\begin{equation}\label{psi_infinite_deterministic}
\tilde{\psi}_\infty (s)=\tilde{\psi}_\infty (s;\alpha)=[ \nbE\left( e^{<s,{\cal Z}_\infty>}\nbu_{[{\cal X}^x_\infty=y]}\right)]_{(x,y)\in {\cal S}^2}=\bigg[\prod_{m=0}^{\infty} \tilde{Q}(s,m) \bigg]',\quad s\in\nbR^k.
\end{equation}
\end{theorem}
\begin{proof}
Since $T_m=m\in\nbN$ and $N_t=t\in\nbN$, \eqref{transient_deterministic_arrival} is a straightforward application of \eqref{eq_psi_tilda} in Proposition \ref{prop_psi_tilda}. 

We recall that $s=(s_1,...,s_k)$ may belong to the set $\nbS:=\{(s_1,...,s_k)\in\nbC^k|\ s_j\in i\nbR,\ j=1,...,k\}$ mentioned shortly after Definition \eqref{def_mgf}, in such as $\tilde{\psi}(s,t)$ is the characteristic function of $\tilde{Z}(t)$ jointly to $X_{N_t}$. Then, in order to prove the convergence in distribution of $(\tilde{Z}(t),X_{N_t})$ given $X_0=x$,  it suffices by L\'evy's convergence theorem to show that $\tilde{\psi}(s,t)$ in \eqref{transient_deterministic_arrival} converges towards $\tilde{\psi}_\infty (s)$ given in \eqref{psi_infinite_deterministic} for all $s\in\nbS$ and $\tilde{\psi}_\infty (s)$ is continuous at $s={\bf 0}$. This part constitutes the main bulk of the proof of the theorem.
%it suffices to prove that $\tilde{\psi}(s,t)$ in \eqref{transient_deterministic_arrival} converges uniformly as $t\to\infty$ while $s$ lies in some neighborhood of ${\bf 0}$ towards $\tilde{\psi}_\infty (s)$ given in \eqref{psi_infinite_deterministic}, as indeed this yields continuity of $\tilde{\psi}_\infty (s)$ at $s={\bf 0}$ which is then indeed an mgf. %So, let $[-M,M]^k$ be some arbitrary neighborhood of ${\bf 0}$ for some $M>0$. One checks easily from \eqref{def_pi_Q_tilda} that $\tilde{\pi}(s,t-1)\longrightarrow I$ uniformly on $s\in [-M,M]^k$ as  $t\to\infty$, so that it is enough to prove that the matrix $\prod_{i=0}^{t} \tilde{Q}(s,i)$ converges uniformly on $[-M,M]^k$ as  $t\to\infty$. To prove this, we use \cite[Section 8.10, p.127]{W34} modified to suit our needs.
Let $||.||$ be a submultiplicative norm on ${\cal S}\times {\cal S}$ matrices, i.e. such that $||MN|| \le ||M||.||N||$ for all matrices $M$ and $N$. % defined by $$||M||:=\sup\left\{\left. \sqrt{x'M'Mx}\; \right| \ {x'x=1,\ x\in \nbR^{\cal S}}\right\}.$$
We write from \eqref{def_pi_Q_tilda1} that
$$
\tilde{Q}(s,m)=P' + (\tilde{\pi}(s,m)-I)P'.
$$
To apply the result given in \cite{A86}, we first introduce the two following norms defined respectively on complex valued and matrices valued sequences (defined similarly as in \cite{A86})
$$
| (u_m)_{m\in\nbN}|_E:=\sum_{m=0}^\infty |u_m|,\quad || (M_m)_{m\in\nbN}||_E=\sum_{m=0}^\infty ||M_m||,
$$
where $u_m\in\nbR$ and $M_m$ is an ${\cal S}\times {\cal S}$ matrix for all $m\in\nbN$. Let us set  $A_m=A_m(s):=(\tilde{\pi}(s,m)-I)P'$ (so as to comply with the notation of the latter paper). Since  $(P')^m$ converges towards $({\bf 1}\pi)'$ as $m\to\infty$, from \cite[Theorem 2.1]{A86} it is sufficient to prove that $|| (A_m)_{m\in\nbN}||_E=\sum_{m=0}^\infty ||A_m||<+\infty$ for the existence of $\prod_{m=0}^{\infty} \tilde{Q}(s,m)$. One has 
\begin{equation}\label{Ai}
||A_m||\le ||\tilde{\pi}(s,m)-I ||.||P'||
\end{equation}
with $\tilde{\pi}(s,m)-I$ a diagonal matrix of which the $(x,x)$th component, $x\in {\cal S}$, is given from \eqref{def_pi_Q_tilda} by $\nbE ( \exp ( \sum_{j=1}^k s_jx_j e^{-\alpha (L_j-m)}\nbu_{[L_j>m]}))-1$. Using the inequality $ |e^u-1|\le e^{|u|}-1$ for all $u\in\nbC$, and remembering that $x_j\in \{0,...,K\}$ is non negative for all $j=1,...,k$, one finds for all $m\in\nbN$ that
\begin{eqnarray}
 & &\bigg|\nbE \bigg( \exp\bigg\{\sum_{j=1}^k s_jx_j e^{-\alpha (L_j-m)}\nbu_{[L_j>m]}\bigg\}\bigg)- 1 \bigg|  \le   \nbE \bigg( \bigg|  \exp\bigg\{\sum_{j=1}^k s_jx_j e^{-\alpha (L_j-m)}\nbu_{[L_j>m]} \bigg\} -1\bigg|  \bigg)\nonumber\\
 &~~~\le &\nbE \bigg( \exp\bigg\{ \bigg|\sum_{j=1}^k s_j  x_j  e^{-\alpha (L_j-m)}\nbu_{[L_j>m]}\bigg|\bigg\}-1 \bigg) \le  \nbE \bigg( \exp\bigg\{ \sum_{j=1}^k |s_j | x_j  \nbu_{[L_j>m]}\bigg\}\bigg)- 1\nonumber\\
&~~~=& \prod_{j=1}^k \left[ 1+ (e^{|s_j| x_j}-1) \nbP(L_j>m)\right]-1 =\sum_{I\subset \{1,...,k\}} \prod_{\ell \in I}\left[(e^{|s_\ell| x_\ell}-1)\nbP(L_\ell>m)\right],\label{psi_infinite_deterministic1}
\end{eqnarray}
%\marge{this $I$ is confusing with the identiy set?}
where the independence of $L_1,...,L_k$ was used. Note now that for all $I\subset \{1,...,k\}$,
\begin{multline*}
\sum_{m=1}^\infty \prod_{\ell \in I}\left[(e^{|s_\ell| x_\ell}-1)\nbP(L_\ell>m)\right]\le \left(e^{\max(|s_1|,...,|s_k|) .K} - 1\right) ^k \sum_{m=1}^\infty \prod_{\ell \in I}  \nbP(L_\ell>m)\\
= \left(e^{\max(|s_1|,...,|s_k|) .K} - 1\right) ^k \nbE\left( \max_{\ell\in I}L_\ell \right),
\end{multline*}
which is finite thanks to the assumption that $\nbE(L_j)<+\infty$ for all $j=1,...,k$. We thus deduce from \eqref{psi_infinite_deterministic1} that $\sum_{m=0}^\infty ||\tilde{\pi}(s,m)-I ||<+\infty $ and in turn, from (\ref{Ai}) $||(A_m)_{m\in \nbN}||_E=\sum_{m=0}^\infty ||A_m||<+\infty$.

Now it remains to prove that $\psi_\infty(s)$ in \eqref{psi_infinite_deterministic} is continuous at $s={\bf 0}$. Let us first recall the inequality $|e^u-1|\le e |u|$ for all $u\in  \nbC$ such that $|u|\le 1$. If $x=(x_1,...,x_k)$ is in ${\cal S}$, this entails that, for all $j=1,...,k$ and $m\in\nbN$,
\begin{equation}\label{aexp}
\left| \exp\left(   s_jx_j e^{-\alpha (L_j-m)}\nbu_{[L_j>m]}\right) - 1\right|\le e |s_j| x_j e^{-\alpha (L_j-m)}\nbu_{[L_j>m]} 
\end{equation}
for all $s_j$ such that $|s_j| x_j\le 1$, which is satisfied if $s=(s_1,...,s_k)\in [-i/K,i/K]^k\subset \nbS$. Letting $u_m^j(s_j):=\nbE \left( \exp\left(  s_jx_j e^{-\alpha (L_j-m)}\nbu_{[L_j>m]}\right)\right)$, we deduce from \eqref{aexp} that
\begin{equation}\label{psi_infinite_deterministic2}
\begin{array}{rcl}
|u_m^j(s_j)-1|&\le& \nbE(e |s_j| x_j e^{-\alpha (L_j-m)}\nbu_{[L_j>m]} )\le e |s_j| x_j \nbP(L_j>m),\\
|u_m^j(s_j)|&\le &e |s_j| x_j \nbP(L_j>m) +1 \le e+1:=M,
\end{array}
\end{equation}
%which, and using that $\sum_{i=0}^\infty\nbP(L_j>i)=1+ \nbE(L_j)$, yields the bounds
%\begin{equation}\label{psi_infinite_deterministic2}
%\begin{array}{rcl}
%|(u_i^j(s_j)-1)_{i\in\nbN}|_E &\le &e |s_j| x_j (1+ \nbE(L_j)),\\
%|u_i^j(s_j)|&\le &e |s_j| x_j (1+ \nbP(L_j>i)) +1 \le 2e:=M,\quad \forall i\in\nbN,
%\end{array}
%\end{equation}
for all $m\in\nbN$, $s_j\in [-i/K,i/K]$ and $j=1,...,k$. Then it follows from \eqref{psi_infinite_deterministic2} that for all $m\in\nbN$ and $s=(s_1,...,s_k)\in [-i/K,i/K]^k$, again by independence of $L_1,...,L_k$,
\begin{eqnarray*}
&&\bigg|\nbE \bigg( \exp\bigg\{ \sum_{j=1}^k s_jx_j e^{-\alpha (L_j-m)}\nbu_{[L_j>m]}\bigg\}\bigg)-1\bigg|=\bigg| \prod_{j=1}^k u_m^j(s_j) -1\bigg| = \bigg|\sum_{r=1}^k \bigg[\prod_{j=1}^{r-1} u_m^j(s_j)\bigg]\left[ u_m^r(s_r)-1\right]\bigg| \\
&\le& \sum_{r=1}^k  \prod_{j=1}^{r-1} \left|  u_m^j(s_j) \right| \left| u_m^r(s_r)-1 \right|
\le \sum_{r=1}^k M^{r-1} e |s_r| x_r \nbP(L_r>m),
\end{eqnarray*}
which, summing from $m=0$ to $+\infty$, yields the following bound for the $|.|_E$ norm for all $x=(x_1,...,x_k)\in {\cal S}$
$$
\bigg|\bigg(\nbE \bigg( \exp\bigg\{ \sum_{j=1}^k s_jx_j e^{-\alpha (L_j-m)}\nbu_{[L_j>m]}\bigg\}\bigg)-1\bigg)_{m\in\nbN}\bigg|_E \le \sum_{r=1}^k M^{r-1} e |s_r| x_r \left(\nbE(L_r)+1\right)
$$
and $s=(s_1,...,s_k)\in [-i/K,i/K]^k$. The right-hand side of the above inequality tends to $0$ as $s\to {\bf 0}$, $s\in [-i/K,i/K]^k \subset \nbS$, and is valid for all $x=(x_1,...,x_k)\in {\cal S}$. By the definition of matrices $\tilde{\pi}(s,m)$, $m\in\nbN$, this immediately implies that
$$
|| (\tilde{\pi}(s,m)-I)_{m\in\nbN}||_E \longrightarrow 0,\quad s\to {\bf 0},\ s\in \nbS. 
$$
One then deduces from \cite[(2.20) in Corollary 2.1]{A86} that the infinite product $s\mapsto \tilde{\psi}_\infty (s)=\left[\prod_{m=0}^{\infty} \tilde{Q}(s,m) \right]'=\left[  \prod_{m=0}^{\infty} \left( P'+A_m(s)\right)\right]'$ is continuous at $s= {\bf 0}$, $s\in \nbS$. This completes the proof.
%Note that the last bound is uniform in $i\in\nbN$, and that it yields easily
%$$
%\sup_{i\in \nbN}\left|\nbE \left( \exp\left(  s_jx_j e^{-\alpha (L_j-i)}\nbu_{[L_j>i]}\right)\right)- 1 \right|\le e |s_j| x_j, \quad \forall s_j\in [-1/K,1/K],\quad \forall j=1,...,k,
%$$
%which in turns implies that $s_j\in [-1/K,1/K]\mapsto \nbE \left( \exp\left(  s_jx_j e^{-\alpha (L_j-i)}\nbu_{[L_j>i]}\right)\right)$ is continuous at $s_j=0$, uniformly in $i\in\nbN$, hence the $(x,x)$th component $\nbE \left( \exp\left( \sum_{j=1}^k s_jx_j e^{-\alpha (L_j-i)}\nbu_{[L_j>i]}\right)\right)= \prod_{j=1}^k \nbE \left( \exp\left( s_jx_j e^{-\alpha (L_j-i)}\nbu_{[L_j>i]}\right)\right)$ of $\tilde{\pi}(s,i)$ is continuous at $s={\bf 0}$, uniformly in $i\in\nbN$. All in all, we proved that $s\in [-1/K,1/K]^k\mapsto\tilde{\pi}(s,i)$ is continuous at $s={\bf 0}$, uniformly in $i\in\nbN$. We deduce immediately that $\sup_{i\in \nbN} ||\tilde{\pi}(s,i)-I||\longrightarrow 0$ as $s\to {\bf 0}$. For all fixed $\varepsilon>0$, let us thus be $\eta>0$ small enough such that $s\in [-\alpha;\alpha]^k$  implies $\sup_{i\in \nbN} ||\tilde{\pi}(s,i)-I||\le \varepsilon$. By \cite[Corollary 2.1]{A86},
%Following \cite{W34}, let us prove that the matrix sequence $(\tilde{Q}(s,i))_{i\in \nbN}$ is a Cauchy sequence, in the sense that
\end{proof}
One interesting consequence of Theorem \ref{theo_deterministic_psi} is that the limiting mgf is expressed conveniently when the $L_j$'s are bounded by some constant $M$. In that case, one has from \eqref{def_pi_Q_tilda} and (\ref{def_pi_Q_tilda1}) that $\tilde{\pi}(s,r)=I$ and $\tilde{Q}(s,r)=P'$ when $r\ge M$, and we thus obtain the following result for this particular case:
\begin{cor}
Suppose that $\tau=1$ a.s. and rvs $L_j$, $j=1,...,k$, are all upper bounded such that $L_j\le M$ a.s. for some deterministic $M\in\nbN^*$. Then the transient mgf in \eqref{transient_deterministic_arrival} simplifies as
$$
\tilde{\psi}(s,t)=\bigg[ \prod_{m=0}^{M-1} \tilde{Q}(s,m)\ (P')^{t-M} \bigg]',\qquad t\geq M,
$$
and the limiting mgf is given by
$$
\tilde{\psi}_\infty(s)= {\bf 1}\pi \bigg[ \prod_{m=0}^{M-1} \tilde{Q}(s,m)\bigg]'.
$$
\end{cor}

\section{Application: Infinite server queues modulated by an external semi-Markovian process}\label{sec:application}
The model described in Section \ref{sec:model} is flexible enough to study the following process in queueing theory and actuarial science. We consider here a semi-Markov process $\{Y(t),\ t\ge 0\}$ with finite state space $\{1,...,\kappa \}$, jump times $(T_i)_{i\in\nbN}$ such that $(T_i-T_{i-1})_{i\in\nbN^*}$ is iid distributed as $\tau$ with cdf $F$, and the embedded Markov Chain $\{ Y(T_n),\ n\in\nbN\}$ having transition matrix and stationary distribution denoted by $P_Y=(p_Y(\ell, m))_{\ell,m=1,...,\kappa}$ and $\pi_Y=(\pi_Y(\ell))_{\ell=1,...,\kappa}$ respectively. Let us suppose that $\{Y(t),\ t\ge 0\}$ models the arrival of customers or claims, such that the $n$th arriving customer has service time/delay $L^\ast_{n,(Y(T_{n-1}),Y(T_n))}$ where $(L^\ast_{n,(\ell, m)})_{n\in\nbN, \ell,m=1,...,\kappa}$ is an iid sequence of rvs $L^\ast_{n,(\ell, m)}$'s, $n\in\nbN,\ell,m=1,...,\kappa$. We also let $(L^\ast_{(\ell, m)})_{\ell,m=1,...,\kappa}$ a generic corresponding rv with LT denoted as $\scrL_{\ell,m}(u)=\nbE(e^{-u L^\ast_{(\ell,m)}})$ for $u\geq 0$. In other words, if $N_t$ denotes the number of clients arrived by time $t$, the $N_t$th customer has service time which depends on both states of the semi-Markov process at the switching time $T_{N_t}$ and prior to this switching time (i.e. depending on both states $Y(T_{N_t})$ and $Y(T_{N_t}-)$), as illustrated in Figure \ref{fig:modulating}. This model has potential applications in queueing theory where an incoming customer may inspect the state of the environment $Y(T_{N_t}-)$ before deciding to join the queue; in an actuarial setting, there are different reasons for the determination of reporting delay for the IBNR claims. In particular, in this case, this model allows random fluctuations in the underlying delay distribution influenced by external process. For example, policyholder may decide to delay the submission of claims under special circumstances such as the external environment process is in a particular state. Also note that some flexibility for this rv is available, for example one may have $\nbP(L^\ast_{(\ell, m)}=0)>0$, implying that a customer finding the environment in state $\ell$ before it switches to state $m$ decides not to join the queue with positive probability. %\textcolor{blue}{Also, in the case of IBNR claims, some claims do not have any delays to be reported, that is, once they occur they are immediately reported to the insurers, see the discussion concerning the numerical application at the end of Section \ref{subsec:expo_service_times}.}\marge{remove? Also, In Fig 1, can you change to $L^\ast_{(\ell,m)}$?}
\begin{figure}[!hbtp]%
\centering
\includegraphics[scale=0.85]{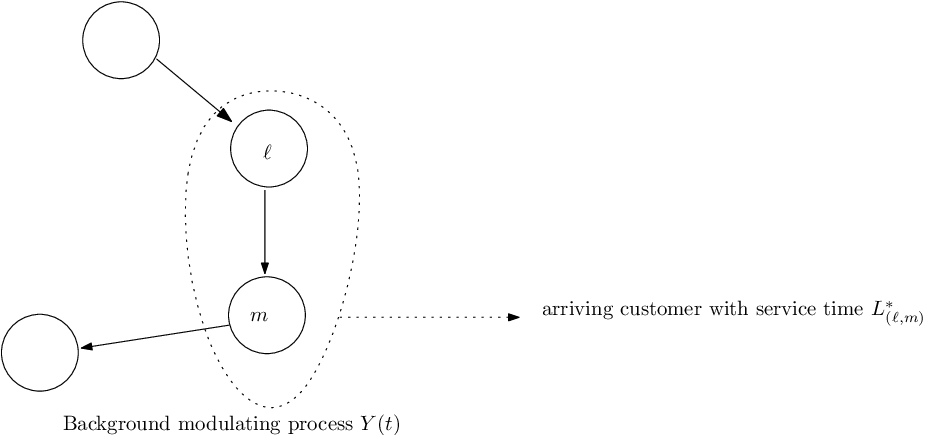}%
\caption{\label{fig:modulating} Modulating semi-Markov process and service time.}
\end{figure}

We then define the corresponding discounted processes $\{\Z(t),\ t\ge 0\}$ and $\{\tilde{\Z}(t),\ t\ge 0\}$ by
\begin{equation}\label{def_semiMarkov}
\Z(t):=\sum_{i=1}^{\infty} e^{-\alpha \left( T_i + L^\ast_{i,(Y(T_{i-1}),Y(T_i))}\right)}\nbu_{\left[ T_i\le t <  T_i + L^\ast_{i,(Y(T_{i-1}),Y(T_i))}\right]},\quad \tilde{\Z}(t)=e^{\alpha t}\Z(t).
\end{equation}
As such, the process defined in \eqref{def_semiMarkov} is different from the one introduced in \eqref{def_Z_t} because the arrival times and the service times are now modulated by some semi-Markov process. However in the following, we shall show that \eqref{def_semiMarkov} is actually embedded into \eqref{def_Z_t}, and this embedding procedure is essentially the central point of this section. Note in particular that this includes the particular case where $\{Y(t),\ t\ge 0\}$ is a continuous-time Markov Chain by considering $\tau\sim {\cal E}(\lambda)$ for some $\lambda>0$, of which infinitesimal generating matrix is given by $Q_Y=(q_Y(\ell ,m))_{\ell,m=1,...,\ka}$ with $q_Y(\ell,m)=\lambda p_Y(\ell,m)$ if $\ell\neq m$ and $q_Y(\ell,\ell)=-\lambda \sum_{m\neq \ell}p_Y(\ell,m)$. In that case, interarrivals may not be identically distributed by considering the generic rv $(L^\ast_{(\ell, m)})_{\ell,m=1,...,\kappa}$ to satisfy $L^\ast_{(\ell, \ell)}=0$ for all $\ell=1,...,\ka$, so that a new customer actually arrives exactly at each time when $Y(t)$ switches to a new state, with service time distributed as $L^\ast_{(\ell, m)}$ when switching from state $\ell$ to state $m$.

So, we need to define the corresponding Markov Chain $(X_i)_{i\in \nbN}$. Let us introduce for all $\ell$ and $m$ in $\{1,...,\kappa\}$ the $\ka \times \ka$ matrix $ e(\ell,m)$ of which the $(j,j')$th entry is $\delta_{(j,j'),(\ell,m)}$. We then define the state space of $(X_i)_{i\in \nbN}$ as
$$
{\cal S}=\left\{ e(\ell,m),\ (\ell,m)\in \{1,...,\kappa\}^2\right\}\subset \{0,1\}^{\ka \times \ka}
$$
so that one sets $k=\ka^2$ and $K=1$, sticking with the notation in Section \ref{sec:model}. Then for all $i\in\nbN$, $X_i=(X_{i, (j,j')})_{(j,j')\in \{1,...,\ka\}^2}$ is of the form $X_i=e(\ell,m)$ for some $\ell$ and $m$ in $\{1,...,\ka\}$, in which case one has
$$
X_{i, (j,j')} = \delta_{(j,j'),(\ell,m)},\quad \forall (j,j')\in \{1,...,\ka\}^2.
$$
The only difference here is in ${\cal S}$ which is a strict subset of $\{0,1\}^{\ka \times \ka}$, however this will not raise any additional technical difficulty in the following analysis. We then define the corresponding transition matrix as $P=(p(x,x'))_{(x,x')\in {\cal S}^2}$ with
\begin{equation}\label{def_transition_matrix_semi_Markov}
p(e(\ell,m), e(\ell',m'))=\left\{
\begin{array}{rl}
p_Y(\ell',m')& \mbox{if } m=\ell',\\
0 & \mbox{otherwise,}
\end{array}
\right.
\quad \ell,\ell',m,m'=1,...,\ka ,
\end{equation}
so that a transition from state $e(\ell,m)$ to state $e(\ell',m')$ of the Markov Chain $(X_i)_{i\in \nbN}$ is only possible if $m=\ell'\iff e(\ell,m) e(\ell',m')=e(\ell,m')$. One checks that \eqref{def_transition_matrix_semi_Markov} defines a proper transition matrix, i.e. the sum over each row is equal to $1$, and $(X_i)_{i\in \nbN}$ is stationary and ergodic iff $\{ Y(T_n),\ n\in\nbN\}$ is, with corresponding stationary distribution $(\pi(x))_{x\in {\cal S}}$ given by
\begin{equation}\label{embed_Markov}
\pi(x)=\pi(e(\ell,m))=p_Y(\ell,m)\pi_Y(\ell),\quad \forall x=e(\ell,m)\in {\cal S},\quad \ell,m=1,...,\ka .
\end{equation}
Finally, we let $(L_{i,(j,j')})_{i\in\nbN, (j,j')\in \{1,...,\ka\}^2}$ a sequence of independent rvs with corresponding distribution given by
$$
L_{i,(j,j')}\sim L^\ast_{i,(j,j')},\quad \forall (j,j')\in \{1,...,\ka\}^2 .
$$
We now arrive at the embedding result. We let $\{Z(t)=Z(t;\alpha)=(Z_{(j,j')}(t))_{(j,j')\in \{1,...,\ka\}^2}\in \nbR^{\{1,...,\ka\}^2} ,\ t\ge 0 \}$ and $\tilde{Z}(t)=e^{\alpha t}Z(t)$ defined by \eqref{def_Z_t}, i.e.
$$
Z_{(j,j')}(t)=\sum_{i=1}^\infty X_{i,(j,j')}e^{-\alpha (L_{i,(j,j')}+T_i)} \nbu_{[T_i \le t < T_i + L_{i,(j,j')}]},\quad (j,j')\in \{1,...,\ka\}^2 .
$$
Then, one checks immediately that the following relation between $\tilde{Z}(t)$ and $\tilde{\Z}(t)$ defined in \eqref{def_semiMarkov} holds
\begin{equation}\label{embedding}
\left\{\tilde{Z}_{(j,j')}(t),\ t\ge 0\right\}\stackrel{D}{=}\left\{\tilde{\Z}(t)\nbu_{[Y(T_{N_t}-)=j, Y(T_{N_t})=j']},\ t\ge 0\right\},\quad \forall (j,j')\in \{1,...,\ka\}^2.
\end{equation}
We remark that the above relation is interesting as it enables us to study $(\tilde{\Z}(t), Y(T_{N_t}-), Y(T_{N_t}) )$ through $\tilde{Z}(t)$ i.e. through the analysis developed in Section \ref{general_results}.
%s \ref{sec:general_results}, \ref{sec:first_second_workload} and \ref{sec:particular_cases}. 
More precisely, one checks that
\begin{eqnarray}
&&\left(\M_{(j_0,j_1), (j_2,j_3)}^{(1)}(t)\right)_{(j_0,j_1)\in \{1,...,\ka\}^2, (j_2,j_3)\in \{1,...,\ka\}^2}\nonumber\\
&:=& \left( \nbE \left( \left.\tilde{\Z}(t) \nbu_{[Y(T_{N_t}-)=j_2, Y(T_{N_t})=j_3]} \right|\ Y(T_{-1})=j_0, Y(T_{0})=j_1\right)\right)_{(j_0,j_1)\in \{1,...,\ka\}^2, (j_2,j_3)\in \{1,...,\ka\}^2}\nonumber\\
 &=&\left(  \nbE \left( \left.\tilde{Z}_{(j_2,j_3)}(t)\right| X_0= e(j_0,j_1) \right) \right)_{(j_0,j_1)\in \{1,...,\ka\}^2, (j_2,j_3)\in \{1,...,\ka\}^2}\nonumber\\
 &=& \left(  M_{(j_2,j_3)}(t){\bf 1} \right)_{(j_2,j_3)\in \{1,...,\ka\}^2}\label{embed_M1}
\end{eqnarray}
where $M_{j}(t)$ is defined in \eqref{def_moments}. Similarly, one can consider for all $j_0$, $j_1$, $j_2$ and $j_3$ the second moment $\M_{(j_0,j_1), (j_2,j_3)}^{(2)}(t):=\nbE \left( \left.\tilde{\Z}(t)^2 \nbu_{[Y(T_{N_t}-)=j_2, Y(T_{N_t})=j_3]} \right|\ Y(T_{-1})=j_0, Y(T_{0})=j_1\right)$, which verifies
\begin{equation}\label{embed_M2}
\left(\M_{(j_0,j_1), (j_2,j_3)}^{(2)}(t)\right)_{(j_0,j_1)\in \{1,...,\ka\}^2, (j_2,j_3)\in \{1,...,\ka\}^2}=\left(  M_{(j_2,j_3),(j_2,j_3)}(t){\bf 1} \right)_{(j_2,j_3)\in \{1,...,\ka\}^2},
\end{equation}
where $M_{jj'}(t)$ is also defined in \eqref{def_moments}. Also, for the workload ${\cal D}(t)$ defined as
$$
{\cal D}(t):= \sum_{i=1}^{\infty} \left(T_i + L^\ast_{i,(Y(T_{i-1}),Y(T_i))}-t\right)\nbu_{\left[ T_i\le t <  T_i + L^\ast_{i,(Y(T_{i-1}),Y(T_i))}\right]}
$$
(which is related with $D(t)$ in Section \ref{sec:workload} as similar to \eqref{embedding}, i.e. its expectation jointly to $(Y(T_{N_t}-), Y(T_{N_t}))$ is given by
\begin{eqnarray*}
&&\left( {\cal W}_{(j_0,j_1), (j_2,j_3)}(t)\right)_{(j_0,j_1)\in \{1,...,\ka\}^2, (j_2,j_3)\in \{1,...,\ka\}^2}\\
&:=&\left( \nbE \left( \left.{\cal D}(t)\nbu_{[Y(T_{N_t}-)=j_2, Y(T_{N_t})=j_3]} \right|\ Y(T_{-1})=j_0, Y(T_{0})=j_1\right)\right)_{(j_0,j_1)\in \{1,...,\ka\}^2, (j_2,j_3)\in \{1,...,\ka\}^2}\\
&=&\left(  W_{(j_2,j_3)}(t){\bf 1} \right)_{(j_2,j_3)\in \{1,...,\ka\}^2},
\end{eqnarray*}
where $W_{j}(t)$ is defined in \eqref{Wit}. 

Furthermore, one could find the following relation between the joint mgf of $\tilde{\Z}(t)$ and distribution of $(Y(T_{N_t}-),Y(T_{N_t}))$ defined for all $z\in \nbR$ and $t\ge 0$ by
$$
\left[ \tilde{\Psi}(z,t)\right]_{(j_0,j_1), (j_2,j_3)}=\nbE \left( \left. e^{z \tilde{\Z}(t)}\nbu_{[Y(T_{N_t}-)=j_2,Y(T_{N_t})=j_3]}\right| Y(T_{-1})=j_0,\ Y(T_0)=j_1 \right),
$$
where $j_0$, $j_1$, $j_2$ and $j_3$ are in $\{1,...,\ka \}$. Then one notices that the above mgf is linked to the joint mgf of $\tilde{\psi}(s,t)$ of $\left\{(\tilde{Z}_{(j,j')}(t))_{j,j'=1,...,\ka},\ t\ge 0\right\}$ thanks to \eqref{embedding} by the relation
\begin{equation}\label{rel_embed_psi}
\left[ \tilde{\Psi}(z,t)\right]_{(j_0,j_1), (j_2,j_3)}= \left[ \tilde{\psi}(z. e(j_2,j_3),t)\right]_{(j_0,j_1), (j_2,j_3)},
\end{equation}
where we recall that $e(j_2,j_3)$ is the $\{1,...,\ka \}\times \{1,...,\ka \}$ matrix of which the $(j,j')$th entry is $\delta_{(j,j'),(j_2,j_3)}$.

The embedding relations \eqref{embed_M1}, \eqref{embed_M2} and \eqref{rel_embed_psi} thus enable us to derive results on the present model with semi-Markovian input, in particular thanks to Sections \ref{sec:first_second_workload} and \ref{sec:particular_cases}. Let us present some of them in the following. For example, \eqref{embed_M1} together with Proposition \ref{prop_asymptotic_moment_general_vector} yields the asymptotic first moment given by
\begin{multline*}
\left(\M_{(j_0,j_1), (j_2,j_3)}^{(1)}(t)\right)_{(j_0,j_1)\in \{1,...,\ka\}^2, (j_2,j_3)\in \{1,...,\ka\}^2}  \\
\longrightarrow  \left( \frac{1}{\nbE(\tau)}\frac{1-\scrL_{(j_2,j_3)}(\alpha)}{\alpha}\ \nbE(X_{(j_2,j_3)}){\bf 1}\right)_{(j_2,j_3)\in \{1,...,\ka\}^2},
\end{multline*}
as $t\to\infty$, where $ \nbE(X_{(j_2,j_3)})=\pi(e(j_2,j_3))=p_Y(j_2,j_3)\pi_Y(j_2)$ by \eqref{embed_Markov}. When service times are exponentially distributed with $L^\ast_{(j_2,j_3)}\sim {\cal E}(\mu_{(j_2,j_3)})$,  then \eqref{embed_M2} and Theorem \ref{theo_second_limit_moment_expo} results in
\begin{multline*}
\left(\M_{(j_0,j_1), (j_2,j_3)}^{(2)}(t)\right)_{(j_0,j_1)\in \{1,...,\ka\}^2, (j_2,j_3)\in \{1,...,\ka\}^2}\longrightarrow\\
\frac{1}{\nbE(\tau)} \bigg[\frac{1-{\cal L}^\tau (\mu_{(j_2,j_3)})}{\mu_{(j_2,j_3)}+2\alpha}\bigg] {\bf 1}\pi\Delta_{(j_2,j_3)}^2 P \ \left( I-{\cal L}^\tau(\mu_{(j_2,j_3)} )P\right)^{-1} \\
+ \frac{2 {\cal L}^\tau (\mu_{(j_2,j_3)} ) }{\nbE(\tau)}\bigg[ \frac{1- {\cal L}^\tau (2\mu_{(j_2,j_3)} )}{2\mu_{(j_2,j_3)} }\bigg]
\bigg(\frac{\mu_{(j_2,j_3)}}{\mu_{(j_2,j_3)}+\alpha} \bigg)^2{\bf 1}\pi\Delta_{(j_2,j_3)} P
\left( I-{\cal L}^\tau(\mu_{(j_2,j_3)} )P\right)^{-1}\\
.\Delta_{(j_2,j_3)} P \  \left( I-{\cal L}^\tau(2\mu_{(j_2,j_3)} )P\right)^{-1},
\end{multline*}
$t\to\infty$. When $\tau\sim {\cal E}(\lambda)$, i.e. when arrivals occur according to a Poisson process and the model is Markov modulated, the transient moment is explicit thanks to \eqref{exact_expression_Mi_Poisson} in Theorem \ref{prop_exact_expression_Mi_Poisson} and one computes easily for all $j_2$ and $j_3$, using $PP'\Delta_\pi {\bf 1}={\bf 1}$, $e^{\lambda v P}{\bf 1}=e^{\lambda v}{\bf 1}$, that
\begin{multline*}
\left(\M_{(j_0,j_1), (j_2,j_3)}^{(1)}(t)\right)_{(j_0,j_1)\in \{1,...,\ka\}^2}=M_{(j_2,j_3)}(t){\bf 1}\\
=\lambda e^{\lambda t (P-I)} \int_0^t   \nbE \left( e^{-\alpha (L_j-v)}\nbu_{[L_j>v]}\right)  e^{-\lambda v (P-I)}   \Delta_j  P e^{\lambda v (P-I)}dv {\bf 1},\quad t\ge 0 .
\end{multline*}
Further, when interarrival are deterministic equal to $1$, then \eqref{rel_embed_psi} and Theorem \ref{theo_deterministic_psi} entail that the limiting joint mgf is given for all $(j_0,j_1)$, $(j_2,j_3)$ in $\{1,...,\ka \}^2$, by
$$
\lim_{t\to \infty}\big[ \tilde{\Psi}(z,t)\big]_{(j_0,j_1), (j_2,j_3)}=\bigg[\prod_{i=0}^{\infty} \tilde{Q}(z. e(j_2,j_3),i) \bigg]_{(j_0,j_1), (j_2,j_3)}'.
$$

\section*{Acknowledgements}
The authors would like to thank two anonymous reviewers for their helpful comments and suggestions. This work was supported by Joint Research Scheme France/Hong Kong Procore Hubert Curien grant No 35296, F-HKU710/15T, and the UNSW Business School 2018 International Research Collaboration Travel Funds. Also, Jae-Kyung Woo gratefully acknowledges the support from and 2019 Business School Research Grant.

\bibliographystyle{alpha}

%\section*{Appendix}

%\subsection*{A. }

\end{document}